\documentclass{article}
\usepackage[utf8]{inputenc}

\title{Stability and Inference of the Euler Characteristic Transform}

\date{}

\usepackage{geometry}
\usepackage{amsmath}
\usepackage{amsfonts, amsthm, amssymb}
\usepackage{xcolor}
\usepackage{mathtools}
\usepackage{comment}

\renewcommand{\epsilon}{\varepsilon}
\newcommand{\ECT}{\mathrm{ECT}}
\newcommand{\SECT}{\mathrm{SECT}}
\DeclareMathOperator{\rank}{\mathrm{rank}\,}
\DeclareMathOperator{\im}{\mathrm{im}\,}

\newcommand{\funcspacer}{\mathcal{F}^r(Z^*,d)}
\newcommand{\homspacer}{\mathcal{E}^r(Z^*,d)}
\newcommand{\imspacer}{\mathcal{G}^r(Z^*,d)}

\DeclareMathOperator{\RR}{\mathbb R}
\DeclareMathOperator{\NN}{\mathbb N}
\DeclareMathOperator{\ZZ}{\mathbb Z}

\newtheorem{lemma}{Lemma}
\newtheorem{theorem}[lemma]{Theorem}

\newtheorem{proposition}[lemma]{Proposition}
\theoremstyle{definition}
\newtheorem{definition}[lemma]{Definition}
\theoremstyle{remark}
\newtheorem*{remark}{Remark}

\usepackage[date=long, giveninits=true, terseinits=true, maxnames=10]{biblatex}
\bibliography{report.bib}

\newcommand{\footremember}[2]{%
    \footnote{#2}
    \newcounter{#1}
    \setcounter{#1}{\value{footnote}}%
}
\newcommand{\footrecall}[1]{%
    \footnotemark[\value{#1}]%
} 
\newcommand\blfootnote[1]{%
  \begingroup
  \renewcommand\thefootnote{}\footnote{#1}%
  \addtocounter{footnote}{-1}%
  \endgroup
}

\author{Lewis Marsh\footremember{maths}{Mathematical Institute, University of Oxford, Oxford, UK}\footremember{ludwig}{Ludwig Institute for Cancer Research, University of Oxford, Oxford, UK} and David Beers\footrecall{maths}}

\usepackage[colorlinks=true, allcolors=blue]{hyperref}

\begin{document}

\maketitle
\blfootnote{Corresponding Email: beers@maths.ox.ac.uk}

\begin{abstract}
    The Euler characteristic transform (ECT) is a signature from topological data analysis (TDA) which summarises shapes embedded in Euclidean space. Compared with other TDA methods, the ECT is fast to compute and it is a sufficient statistic for a broad class of shapes. However, small perturbations of a shape can lead to large distortions in its ECT. In this paper, we propose a new metric on compact one-dimensional shapes and prove that the ECT is stable with respect to this metric. Crucially, our result uses curvature, rather than the size of a triangulation of an underlying shape, to control stability. We further construct a computationally tractable statistical estimator of the ECT based on the theory of Gaussian processes. We use our stability result to prove that our estimator is consistent on shapes perturbed by independent ambient noise; i.e., the estimator converges to the true ECT as the sample size increases.
\end{abstract}

\section{Introduction}

Classifying shapes is a ubiquitous task in data science and machine learning. A wealth of theory has been developed to distinguish different shapes and a large array of applications of these methods in the natural sciences exist \cite{boyer2011algorithms, donnat2022deep, gao2019gaussian, wang2021statistical}. In particular the Euler characteristic transform (ECT) \cite{turner2014persistent}, arising from topological data analysis (TDA), provides a sufficient statistic for a large class of shapes \cite{curry2018many, ghrist2018persistent} (e.g., compact semi-algebraic sets) embedded in Euclidean space by considering intersections of a shape with half-spaces. Each half-space in $\RR^d$ can be associated with an integer by computing the Euler characteristic of the part of the shape that lies within the given half-space. By this process, a map is induced from half-spaces in $\RR^d$ to $\ZZ$ which is defined to be the ECT. Typically, the ECT is viewed as a function from $S^{d-1}\times \RR$ to $\ZZ$, via an identification of half-spaces in $\RR^d$ to $S^{d-1}\times \RR$. Related to the ECT is the smooth Euler characteristic transform (SECT), which encodes the same information as the ECT but is a continuous function.

The ECT is theoretically well-motivated, interpretable and has been successfully applied in practice \cite{amezquita2020quantifying, crawford2020predicting, marsh2022detecting, nadimpalli2023euler, wang2021statistical}. 
Further, the ECT signature lies in a vector space of functions and is thus well-suited for further statistical analysis. While the ECT is fast to compute, small perturbations in the input shape can lead to large differences in the output signature \cite{chevyrev2018persistence}. By contrast to other TDA methods, such as persistent homology~\cite{cohen2005stability} and the persistent homology transform \cite{turner2014persistent}, we are not aware of any general stability results for the ECT which are independent of the triangulation of a shape.

\subsection{Contributions}
We propose a new metric on the embeddings of a finite one-dimensional CW complex that is sensitive to changes in arc-length. Next, we introduce a norm on Euler characteristic transforms, in a similar vein to the norm introduced in Meng et al. \cite[Equation 3.1]{meng2022randomness}, defined by taking first the 1-norm over the $\RR$ component, and then the $\infty$-norm over the $S^{d-1}$ component.
We then prove a novel stability result for the ECT, showing that the ECT is continuous in our metric of embedded spaces (Theorem~\ref{thm:ECTstab}). In other words, if two embeddings of the same one-dimensional CW complex are sufficiently close in our metric, their corresponding ECTs are also close. To the best of our knowledge, our result is the first stability result for the ECT which is independent of the triangulation of a shape. Using similar ideas, we also show that the ECT of a smooth underlying shape can be approximated using sufficiently fine triangulations (Theorem \ref{thm:interpolationstable}). Further, we propose a smoothing method for embeddings of one-dimensional CW complexes that were perturbed by independent Gaussian noise in ambient space. We use the two previous results to prove that our smoothing method does not only yield stability but also provides a consistent statistical estimator for the ECT of a noisy data set (Theorem \ref{thm:combined}), i.e. the ECT of the smoothed shape converges to the ECT of the underlying shape in probability as we increase the number of noisy observations.

The well-known stability results in applied topology for \v{C}ech and Vietoris-Rips filtrations of point clouds are stated in terms of the Hausdorff distance \cite{chazal2014persistence}. Proving stability results for the ECT is complicated by the fact that this metric is too coarse for the ECT to be continuous.
Crucially, it is straightforward to construct examples of two shapes embedded in Euclidean space which are close in Hausdorff distance but whose ECTs are far apart. We loosely classify such instabilities into two categories.

The first type of instability arises when two shapes are close in Hausdorff distance, yet not homeomorphic to each other. Counterexamples can be constructed by adding a single point to a shape at an arbitrarily close distance, as visualised in Figure \ref{fig:instability_type1} for the case of an embedded simplicial complex. We point out that classical persistent homology and the persistent homology transform (PHT) \cite{turner2014persistent} suffer from the same instability. However, extended persistence \cite{cohensteiner2008extending} and the extended persistent homology transform \cite{turner2022extended} can be used to partially overcome this type of instability. In this paper, we resolve the described type of instability by restricting ourselves to shapes that are homeomorphic. Restricting an ECT analysis to a homeomorphism class of shapes is common in applications, see for example \cite{amezquita2020quantifying, marsh2022detecting, nadimpalli2023euler, tang2022topological}.

\begin{figure}
    \centering
    \includegraphics[width=0.7\textwidth]{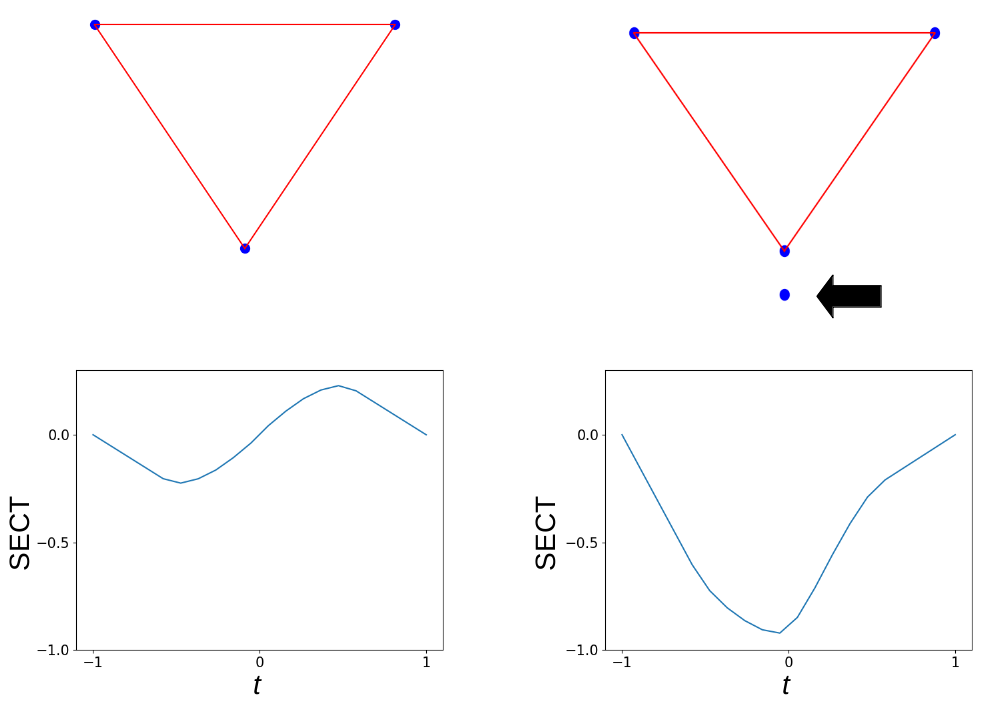}
    \caption{We visualise two embedded simplicial complexes (top) which differ by a single vertex. Their SECTs (bottom), visualised for a filtration in the bottom-to-top direction, are significantly different. The illustrated behaviour persists when we move the disconnected vertex in the top right panel (indicated by arrow) arbitrarily close to the larger connected component.}
    \label{fig:instability_type1}
\end{figure}

Secondly, the ECT can suffer from instability through excessive curvature. For example, in the case of shapes homeomorphic to $S^1$ or $I=[0,1]$, which can be parametrised as curves, this type of instability occurs when two curves are close in the embedded space, but one curve changes curvature much more rapidly. An example of such curves is given in Figure \ref{fig:instability_type2}. Such instability is expected to occur if a shape is approximated based on points which are perturbed \emph{independently} of each other by ambient noise.

\begin{figure}
    \centering
    \includegraphics[width=0.7\textwidth]{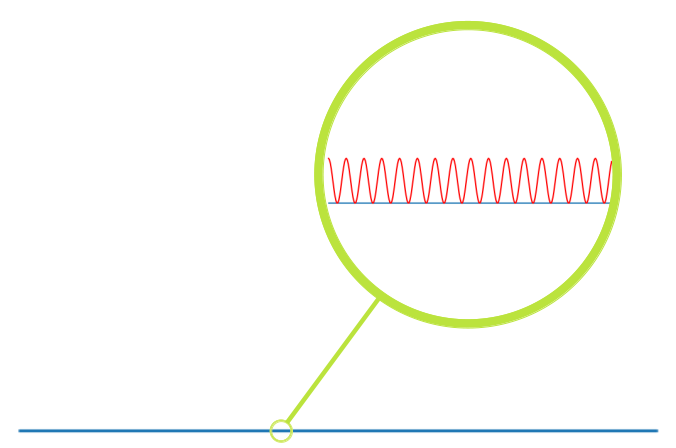}
    \caption{Two shapes homeomorphic to $[0,1]$ embedded into $\RR^2$: A straight line (blue) and a wave (red) closely following the straight line with a small amplitude $\varepsilon$ and high frequency. As long as the frequency is high enough for the wave to go through $n:=\lceil 1/\varepsilon\rceil$ amplitudes, the distance of the ECTs of the two curves is at least 1 (fix the $S^1$-component of the ECTs to be the bottom-up direction and compute the 1-norm over $\RR$), while the Hausdorff distance between the curves is $\varepsilon$.}
    \label{fig:instability_type2}
\end{figure}

Our work resolves instabilities of the second type for one-dimensional shapes by proposing a new metric which is sensitive to curvature. We also provide a statistical estimator of the ECT which is consistent under perturbations by independently distributed Gaussian noise. These perturbations are likely to produce changes in curvature. While the PHT and extended PHT do not suffer from instabilities as illustrated in Figure \ref{fig:instability_type2}, neither method provides a consistent estimator. Furthermore, the ECT arguably provides signatures more amenable to the application of further statistical and machine learning methods and are, by themselves as well as in conjunction with our new method, faster to compute. 

\subsection{Related Work}
We note that the stability of the Wasserstein distance proved by Skraba and Turner \cite{skraba2020wasserstein} provides a straightforward stability result for the ECT. Further, D{\l}otko and Gurnari \cite{dlotko2022euler} prove a similar result for the Euler characteristic curve. Nadimpalli et al. \cite{nadimpalli2023euler} prove a stability result for the ECT on binary image data, which is linear in the number of voxels at which two images differ. However, these stability results depend on the number of simplices in the underlying simplicial complex and the bound on the ECT becomes increasingly loose as the number of data points increases.  Meng et al. \cite{meng2022randomness} provide results that imply stability of the ECT when a shape is perturbed by rotations and translations but not for more general perturbations. Tameness assumptions, which are not needed in our results, are required for the stability they prove to hold. They also provide a statistical inference pipeline for shapes using the SECT. However, their pipeline considers parameterised families of shapes and random perturbations only happen in parameter space. As a result, the perturbations of points in shape space are correlated. By contrast, our results on the estimation of the ECT and SECT allow independent perturbations in ambient space.

\subsection{Outline}

The paper is structured as follows: We start by introducing background on one-dimensional CW complexes, the ECT and SECT in Section~\ref{sec:background}. Further, we introduce Gaussian processes. In Section \ref{sec:ECTstab} we propose a novel metric on the space of embeddings of a finite one-dimensional CW complex and prove that the ECT is stable against this metric for $C^2$-embeddings in Theorem~\ref{thm:ECTstab}. Then we propose a method for approximating the ECT of such an embedding by interpolating points in a finite subset in Theorem~\ref{thm:interpolationstable}. In Theorem \ref{thm:smoothing} of Section \ref{sec:ect_prob_stab} we prove the probabilistic convergence of Gaussian processes on finite one-dimensional CW complexes, given a suitable kernel. Next, we construct a statistical estimator of the ECT for a shape perturbed by independent Gaussian noise. In Theorem \ref{thm:combined} we combine our deterministic stability results with our probabilistic convergence result to prove that the estimator is consistent. Finally, we illustrate the power of our estimator and results on an example in Section~\ref{sec:example}.

\section{Background}\label{sec:background}

\subsection{Topological Preliminaries}

\subsubsection{One-Dimensional CW Complexes}
A topological space $Z$ is called a one-dimensional CW complex if it is of the form
\begin{equation}
\label{eqn:1dimCWdef}
    Z  = \bigg(Z_0 \sqcup \bigsqcup_{\lambda \in \Lambda} [0,1]\bigg)/\phi,
\end{equation}
where $Z_0$ is a set with the discrete topology and $\phi$ is some map from the endpoints of the intervals in $\bigsqcup_{
\lambda \in \Lambda} [0,1]$ to $Z_0$. We refer to the map sending the $\lambda^\textrm{th}$ copy of $[0,1]$ into $Z$ by $\Phi_\lambda$ (note that $\Phi_\lambda$ must be injective everywhere except possibly the endpoints of the interval). The space $Z$ is said to be a finite one-dimensional CW complex if it can be written as in Equation (\ref{eqn:1dimCWdef}) with $Z_0$ and $\Lambda$ finite. We refer to points in $Z$ that are in the image of $Z_0\to Z$ as 0-cells and subsets of $Z$ that are the image of a map $\Phi_\lambda$ as 1-cells. For convenience, we denote the set of 1-cells of $Z$ by $Z_1$. It may be possible for a space $Z$ to be written as in Equation (\ref{eqn:1dimCWdef}) in many different ways. For instance, the circle admits the structure of a one-dimensional CW complex with $n$ 0-cells and $n$ 1-cells for any $n$. Sometimes, we need to fix a cellular decomposition of a shape. When fixing a choice of $Z_0$ and $\{\Phi_\lambda\}_{\lambda\in\Lambda}$, we refer to $Z^* = (Z, Z_0,\{\Phi_\lambda\}_{\lambda\in\Lambda})$ as a CW structure on $Z$.

We are primarily interested in shapes with this structure that are subsets of $\mathbb{R}^d$. To this end, we say that $f:Z\to X\subseteq \mathbb{R}^d$ is a $C^r$ map under $Z^*$ if $f$ is $C^r$ on its restriction to each copy of $[0,1]$ in Equation (\ref{eqn:1dimCWdef}). We denote the set of such maps $f$ by $\funcspacer$. We denote the subset of $\funcspacer$ of maps that are also homeomorphisms by $\homspacer$ and the set of images of these homeomorphisms by $\imspacer$.

For $r \geq 2$, we say that $f \in \funcspacer$ has curvature bounded by $M$ if the curvature of the map $[0,1] \to Z \to X$ is bounded by $M$ for every copy of $[0,1]$ in Equation (\ref{eqn:1dimCWdef}). By compactness of the unit interval, it follows that every $f\in \imspacer$ has curvature bounded by some constant $M$ whenever $Z$ is a finite one-dimensional CW complex and $r \geq 2$. We say $X \in \imspacer$ has curvature bounded by $M$ under $Z^*$ if the curvature of any map $h\in\homspacer$ with image $X$ has curvature bounded by $M$. It is straightforward to show that if $X$ has curvature bounded by $M$ under $Z^*$, then every map $h\in\homspacer$ with image $X$ has curvature bounded by $M$.

\subsubsection{The Euler Characteristic Transform}

The Euler characteristic of a topological space $X$ with finitely generated homology is defined as the following alternating sum, which is homotopy invariant
\begin{equation*}
    \chi(X) : =  \sum_{k = 0}^\infty (-1)^k \rank H_i(X;\mathbb{Z}).
\end{equation*}
Here, the rank of a finitely generated abelian group is the number of $\mathbb{Z}$ summands in its canonical decomposition. If $X$ is homotopy equivalent to a finite one-dimensional CW complex containing $c_k$ cells of dimension $k$ for $k=0,1$, then the following alternate formula for $\chi(X)$ also holds
\begin{equation*}
    \chi(X) = c_0 -c_1.
\end{equation*}
For a proof of this equation, see for example \cite[Theorem 2.44]{hatcher2002algebraic}. In particular, if every path-component of $X$ is contractible, then $\chi(X)$ is the number of path-components of $X$.

For a subset $X \subseteq \mathbb{R}^d$, we define the \textit{Euler characteristic transform} (ECT) of $X$ to be the following map:
\begin{equation*}
    \begin{split}
        \ECT_{X}: S^{d-1}\times \mathbb{R} &\longrightarrow \mathbb{Z}\\
        (v,t) &\longmapsto \chi(\{x\in X : \langle x, v \rangle \leq t \}).
    \end{split}
\end{equation*}
In words, the Euler characteristic transform of a shape $X$ encodes the Euler characteristic of the intersection of $X$ with every closed half-space with affine boundary. When $\ECT_X(v,t)$ is not well defined, we set $\ECT_X(v,t) = \infty$. In this context, we set $\infty + \infty = \infty$, $\infty-\infty = \infty$, and $\infty + n = \infty$ for any integer $n$.

The Euler characteristic transform has been of interest in the context of studying geometric data since, perhaps surprisingly, the map $X \mapsto \ECT_X$ is injective when $X$ is a simplicial complex. Hence, researchers are able to convert geometric data into functional data, which can be more easily studied on a computer and by methods of classical statistics and machine learning. Injectivity is not difficult to show in dimension one. In dimensions 2 and 3 it was first shown in \cite[Theorem 3.1, Corollary 3.2]{turner2014persistent}. This result was generalised using Euler calculus to any dimension independently in \cite[Theorem 3.5]{curry2018many} and \cite[Theorem 5]{ghrist2018persistent}. In fact, both of these papers show that the ECT is injective on families of constructible sets. In particular, the ECT is injective on compact semi-algebraic sets.

Often one restricts to constructible families of subsets of $\mathbb{R}^d$ when studying the ECT, however, for the results presented in this paper these assumptions are unnecessary and further mention of constructible sets will be limited.

If the subset $X\subset\RR^d$ is bounded, as is the case for any $C^r$-embedding of a finite one-dimensional CW complex, we can restrict the ECT to being a function on $S^{d-1}\times [-a,a]$, where $a\in\RR$ is greater than the bound of $X$ (in the Euclidean norm). For a fixed direction $v\in S^{d-1}$, we then define the Euler characteristic curve (ECC) as $\mathrm{ECC}_{X,v}(t)=\ECT_X(v,t)$ for each $t\in[-a,a]$. The \emph{smooth Euler characteristic transform} (SECT), introduced in \cite{turner2014persistent}, is then defined as
\begin{align*}
    \mathrm{SECT}_X: S^{d-1}\times[-a,a] & \longrightarrow \RR\\
    (v,t) & \longmapsto \int_{-a}^t \ECT_X(v,x)-\overline{\mathrm{ECC}_{X,v}}\,\mathrm{d}x,
\end{align*}
where $\overline{\mathrm{ECC}_{X,v}}$ is the mean of $\mathrm{ECC}_{X,v}$ over $[-a,a]$. The SECT contains the same information on $X$ as the ECT. However, the SECT lies in the Hilbert space $\mathcal{L}_2(S^{d-1}\times[-a,a])$ and is, therefore, more amenable to the application of further statistical methods \cite{wang2021statistical}.

For the remainder of this paper, we endow the ECT (viewed as a function on $S^{d-1}\times [-a,a]$) with the norm
\begin{equation}
    \left\|\ECT_X\right\|:=\sup_{v\in S^{d-1}}\int_{-a}^a|\ECT_X(v,t)|\,\mathrm{d}t.\label{eq:norm}
\end{equation}
The norm is defined and considered analogously for the SECT.

It is also useful for us to define Euler characteristic transforms of functions $f$ from topological spaces into $\mathbb{R}^d$. We define
\begin{equation*}
    \begin{split}
        \ECT_{f}: S^{d-1}\times \mathbb{R} &\longrightarrow \mathbb{Z}\\
        (v,t) &\longmapsto \chi\big(f^{-1}\{x\in \mathbb{R}^d : \langle x, v \rangle \leq t \}\big).
    \end{split}
\end{equation*}
Note that if $f$ is a homeomorphism it is immediate that $\ECT_f = \ECT_{\im f}$. We prescribe norms to Euler characteristic transforms of functions as before: restricting to functions bounded by $a$, we let
\begin{equation*}
    \left\|\ECT_f\right\|:=\sup_{v\in S^{d-1}}\int_{-a}^a|\ECT_f(v,t)|\,\mathrm{d}t.
\end{equation*}

\subsection{Kernel Methods and Gaussian Processes}

Gaussian processes (GPs) are a model for random functions. Before defining GPs, we introduce the notion of a \emph{kernel}:

\begin{definition}
Let $X$ be as set. A kernel on $X$ is a symmetric function $k:X\times X\to\RR$ such that for all $x_1,...,x_n\in X$ and all $a_1,...,a_n\in\RR$ we get
\begin{equation}
    \sum_{i,j=1}^n a_ia_jk\left(x_i,x_j\right)\geq 0.\label{eq:kernel-pd}
\end{equation}
Given finite subsets $X'=\{x'_1,...,x'_m\}$ and $X^*=\{x^*_1,...,x^*_n\}$ of $X$, we denote by $K(X',X^*)$ the $m\times n$ matrix with $i,j$-entry $K(X',X^*)_{ij}=k\left(x'_i,x^*_j\right)$, which is called the \emph{Gram matrix} of $k$ at $X'$ and $X^*$.

To each kernel $k$, we can associate its \emph{reproducing kernel Hilbert space} (RKHS): define a vector space of functions on $X$ by $\mathcal{H}_0=\mathrm{span}_{\RR}\{k(\,\cdot\,,x)\,\vert\,x\in X\}$ with inner-product induced by
\begin{equation}
    \langle k(\,\cdot\,,x),k(\,\cdot\,,y)\rangle:=k(x,y).\label{eq:rep_prop}
\end{equation}
Then the RKHS of $k$ is defined to be $\mathcal{H}_k=\overline{\mathcal{H}_0}$, the completion of $\mathcal{H}_0$.
\end{definition}
Note that Equation (\ref{eq:kernel-pd}) is equivalent to each Gram matrix of the form $K(X',X')$ being positive-definite.

\begin{definition}\label{def:gp_def}
    Let $X$ be a set, $\mu:X\to\RR$ a function and $k:X\times X\to\RR$ be a kernel. The Gaussian process (GP) on $X$ with mean function $\mu$ and kernel $k$ is defined to be the random function $f:X\to\RR$ such that for each finite set $X'=\{x'_1,...,x'_n\}\subseteq X$ we get
    \begin{equation}
        \begin{bmatrix}
        f(x'_1) \\ \vdots \\f(x'_m)
        \end{bmatrix}\sim
        \mathcal{N}\left(\begin{bmatrix}
            \mu(x'_1)\\ \vdots \\ \mu(x'_n)
        \end{bmatrix}, K\left(X', X'\right)\right).
    \end{equation}
\end{definition}

The theory of GPs can be used to estimate a deterministic function $f:X\to\RR$ given noisy observations of $f$ at points $\{x_1,...,x_n\}\subseteq X$. Most commonly this is done by a \emph{Gaussian process regression} (GPR), a nonparametric Bayesian method, which models $f$ as a random function.
When performing a GPR with a given kernel $k$, one typically constructs a \emph{prior distribution} by assuming
\begin{equation}
    \begin{bmatrix}
f(x'_1) \\ \vdots \\f(x'_m)
\end{bmatrix}\sim
\mathcal{N}\left(\mathbf{0}, K\left(X', X'\right)\right) \label{eq:prior}
\end{equation}
for any finite subset $X'\subseteq X$ \cite{rasmussen2003gaussian}. Assume we make $n$ observations of the form $y_i=f(x^*_i)+\zeta_i$, where $\zeta_i\sim\mathcal{N}(0,\sigma^2)$ i.i.d, for $i=1,...,n$ and $\sigma>0$. Importantly, $\zeta_i$ does not depend on $f$ in any way. Then, by our prior assumption and by the introduction of the shorthand $\mathbf{f}(X'):=(f(x_1),...,f(x_m))^T$ and $\boldsymbol\zeta:=(\zeta_1,...,\zeta_m)^T$ we get that
$$
\begin{bmatrix}
\mathbf{f}(X^*) + \boldsymbol\zeta\\
\mathbf{f}(X')
\end{bmatrix}\sim
\mathcal{N}\left(\mathbf{0},
\begin{bmatrix}
K(X^*, X^*) + \sigma^2I_n & K(X^*, X')\\
K(X', X^*) & K(X', X')
\end{bmatrix}
\right).
$$
Thus, by conditioning the above multivariate normal distribution of $\mathbf{f}$ on the observations $\mathbf{y}:=(y_1,...,y_n)^T$, we get \cite{rasmussen2003gaussian}
\begin{align}
\mathbf{f}(X')\vert \mathbf{f}(X^*)+\boldsymbol\zeta=\mathbf{y}\sim\mathcal{N}(&K(X',X^*)(K(X^*,X^*)+\sigma^2I_n)^{-1}\mathbf{y}, \notag\\ & K(X',X')-K(X',X^*)(K(X^*,X^*)+\sigma^2I_n)^{-1}K(X^*,X')). \label{eq:posterior}
\end{align}

The above distribution, for any finite $X'\subseteq X$, is the \emph{posterior distribution} of $f$ given $X'$. In the context of Bayesian modelling, we first summarise our knowledge in the values of $f$ by the prior distribution: unless we gain further information, we assume $f$ to be mean 0 with covariance $K$ (Equation (\ref{eq:prior})). For any (noisy) observation of $f$ we make, we update our belief in the values of $f$ by conditioning our prior distribution on our observations.
The posterior density at inputs $X'$ can then be interpreted as how strongly we believe an output value to be the true output of $f$ at $X'$, given our observations and modelling assumptions.

If $m=1$ (i.e., $X'=\{t\}$ for some $t\in X$), we denote the mean of the above conditional normal distribution by $\hat{f}_n(t)$ and its variance by $v_n(t)$. We henceforth call $\hat{f}_n(t)$ the \emph{Gaussian smoothing of $f$} (on the set $X^*$ of size $n$). When needed, we explicitly denote the dependence of $\hat{f}_n(t)$ on $X^*$ and $f$ by writing $\hat{f}_n(t,X^*,f)$. From Equation (\ref{eq:posterior}), it follows that $\hat{f}_n$  always lies in $\mathcal{H}_k$, the RKHS of $k$.

Under certain assumptions, one can show that $\hat{f}_n\to f$ in mean. In our paper, we use results by Koepernik and Pfaff \cite{koepernik2021consistency}, which give strong probabilistic convergence results in the case of $X$ being a compact metric space.

We note that computing $\hat{f}_n$ requires the inversion of an $n\times n$-matirx and thus has a runtime of $\mathcal{O}(n^3)$. By using the ECT on $\hat{f}_n$ we thus lose some of the ECTs runtime advantage (compared to the PHT and extended PHT). However, both versions of the PHT require $\mathcal{O}(n_s^3)$ computations \emph{per direction} \cite{edelsbrunner2010computational}, where $n_s\geq n$ is the number of simplices in the triangulation of a shape, which still gives a combined Gaussian process and ECT pipeline an edge in terms of runtime. More importantly, it is common practice to approximate the (inverse) Gram matrix by a low-rank matrix approximation method, such as the Nystroem method \cite{williams2000using} or random Fourier features \cite{rahimi2007random}. Such methods run in $\mathcal{O}(l^3+l^2n)$, where $l\ll n$ is the approximate rank of the Gram matrix and thus are significantly faster than $\mathcal{O}(n^3)$.

\section{ECT Stability of Non-Random Data}
\label{sec:ECTstab}

\subsection{Stability for smooth curves}

As we observed in the introduction, controlling the proximity of two different one-dimensional shapes is not enough to control the difference between their ECTs. This motivates the definition of a metric between such shapes which is also concerned with perturbations to length.

\begin{definition}
\label{def:distdef}
    Let $Z$ be a finite one-dimensional CW complex with a fixed CW structure $Z^* = (Z, Z_0, \{\Phi_\lambda\}_{\lambda\in\Lambda})$. Fix $r\geq 1$. For $X,Y \in \imspacer$, we define $d_{Z^*}(X,Y)$ to be the infimum of all $\epsilon$ such that there exists $h_X,h_Y \in \homspacer$, whose images are $X$ and $Y$ respectively, satisfying:
    \begin{enumerate}
        \item The difference of arc lengths between $h_X\circ\Phi_\lambda$ and $h_Y\circ\Phi_\lambda$ is less than or equal to $\epsilon$ for each $\lambda$.
        \item Both $h_X\circ\Phi_\lambda$ and $h_Y\circ\Phi_\lambda$ are curves of constant velocity for each $\lambda$.
        \item $\|h_X - h_Y\|_\infty \leq \epsilon$.
    \end{enumerate}
\end{definition}

By using the compactness of $Z$ it is not difficult to show that $d_{Z^*}$ is a metric on $\imspacer$. For the remainder of this paper, we endow $\imspacer$ with this metric and the topology arising from it. A key goal of this paper is to show that the ECT is a continuous map on $\imspacer$ for $r \geq 2$.

\begin{theorem}
\label{thm:ECTstab}
Let $Z$ be a finite one-dimensional CW complex with a fixed CW structure $Z^* = (Z, Z_0, \{\Phi_\lambda\}_{\lambda\in\Lambda})$. The map $X \mapsto \ECT_X$ is continuous on $\imspacer$ for $r \geq 2$.

In particular, if $X$ has curvature bounded by $M$, and the image of the $\lambda^{\textrm{th}}$ 1-cell in $X$ has arc length $L_\lambda$, then whenever $d_{Z^*}(X,Y) < \epsilon$, we have 
\begin{equation*}
\left\|\ECT_X-\ECT_Y\right\| \leq |Z_0|\epsilon + \sum_{\lambda\in\Lambda} G_\lambda(\epsilon),
\end{equation*}
where
\begin{equation*}
G_\lambda(\epsilon) := 
\begin{cases}
                8\sqrt{L_\lambda n_\lambda\epsilon} + n_\lambda \epsilon &  L_\lambda/n_\lambda > 2\epsilon \\
                11n_\lambda\epsilon & L_\lambda/n_\lambda \leq 2\epsilon
            \end{cases}
\end{equation*}
and
\begin{equation*}
    n_\lambda := \max\left(\left\lceil \left(\frac{M^2L_\lambda^3}{24\epsilon}\right)^{1/3} \right\rceil,\left\lceil \frac{L_\lambda M}{\pi}\right\rceil\right).
\end{equation*}
\end{theorem}

We prove this theorem via the following proposition, which is useful when considering functional information in Section \ref{sec:ect_prob_stab}.

\begin{proposition}
\label{prop:funcstab}
Let $Z$ be a finite one-dimensional CW complex with a fixed CW structure $Z^* = (Z, Z_0, \{\Phi_\lambda\}_{\lambda\in\Lambda})$. Let $f,g \in \funcspacer$, with $r\geq 2$, and suppose that:
\begin{enumerate}
\item The curves $f\circ\Phi_\lambda$ and $g\circ\Phi_\lambda$ have arc lengths that differ by at most $\epsilon$ for each $\lambda$.
\item The curves $f\circ\Phi_\lambda$ and $g\circ\Phi_\lambda$ have constant velocity for each $\lambda$.
\item $\|f - g\|_\infty \leq \epsilon$.
\end{enumerate}

Then if $f$ has curvature bounded by $M$, and $f\circ\Phi_\lambda$ has arc length $L_\lambda$, we have
\begin{equation*}
\big\|\ECT_f-\ECT_g\big\| \leq |Z_0|\epsilon + \sum_{\lambda\in\Lambda} G_\lambda(\epsilon),
\end{equation*}
where $G_\lambda$ and $n_\lambda$ are defined as above.
\end{proposition}

The idea of the proof of this proposition is as follows. We show that the norm of the ECT of a curve can be controlled using its differential properties. Using this observation, we can bound the difference in ECT of two curves that are nearly straight lines. We can also refine the structure of a finite one-dimensional CW complex $Z^*$ with a map $f:Z\to \mathbb{R}^d$ into enough pieces that the image of every 1-cell is a nearly linear curve. The above proposition then follows from a glueing argument.

In this paper we let $I$ denote the unit interval $[0,1]$ and say that $f:I \to \mathbb{R}^d$ is piece-wise $C^1$ if it is continuous and there exists a collection $T_1,\ldots, T_k$ of closed intervals covering $I$, on the interiors of which $f$ is $C^1$.

\begin{proposition}
\label{prop:ectbound}
Let $\gamma:I\to\mathbb{R}^d$ piece-wise $C^1$ map, with $X$ being the image of $\gamma$. Fix any $v\in S^{d-1}$. Let $a$ and $b$ be the minimal and maximal values of $f:x\mapsto \langle v, \gamma(x)\rangle$ on $I$ respectively. Then
\begin{equation*}
    \int_a^b |\ECT_\gamma(v,t)|\;\mathrm{d}t \leq V(f),
\end{equation*}
where $V(f)$ denotes the variation of $f$:
\begin{equation*}
    V(f) := \int_0^1 |f'|\;\mathrm{d}t.
\end{equation*}
For $t\geq b$ we have $\ECT_\gamma(v,t) = 1$. For $t<a$ we have $\ECT_\gamma(v,t) = 0$. In particular, $\ECT_\gamma(v,t)$ is defined almost everywhere.
\end{proposition}

\begin{proof}
The last two statements follow from the fact that $I$ has Euler characteristic one and the empty set has Euler characteristic zero. We now prove the remainder of the proposition.

The main goal of the proof is to establish the following two equalities:
\begin{equation*}
    \begin{split}
        \int_a^b |\ECT_\gamma(v,t)|\;\mathrm{d}t &= \int_a^b\Big|\pi_0\big[f^{-1}(-\infty,t]\big]\Big|\;\mathrm{d}t\\
        V(f) &=\int_a^b\Big|\pi_0\big[f^{-1}(t)\big]\Big|\;\mathrm{d}t.
    \end{split}
\end{equation*}

The first of these equalities follows from the fact that every subset of the unit interval is component-wise contractible. Hence, the Euler characteristic of any subset of $I$ is the number of path-components it has. The second of these equalities is more difficult to show, and its establishment is the bulk of the proof. Once both equalities are shown, demonstrating that the first integrand on the right is less than or equal to the second integrand on the right completes the proof.

To begin, notice that $f$ is piece-wise $C^1$ since $\gamma$ is. For the proof, we let $G(f)$ be the set of points where $f'$ is defined and positive, $D(f)$ be the set of points where $f'$ is defined and negative, and $C(f)$ be the set of points in $I$ neither in $G(f)$ or $D(f)$. Since $f$ is piece-wise $C^1$, $G(f)$ and $D(f)$ are both open. Meanwhile, all but finitely many points of $C(f)$ satisfy $f'(x) = 0$. Clearly, $G(f)$, $D(f)$, and $C(f)$ partition $I$. Hence,
\begin{equation*}
    \begin{split}
        V(f) &= \int_{G(f)}f'(x)\;\mathrm{d}x + \int_{D(f)}-f'(x)\;\mathrm{d}x + \int_{C(f)} |f'(x)|\;\mathrm{d}x \\
        &= \int_{G(f)}f'(x)\;\mathrm{d}x + \int_{D(f)}-f'(x)\;\mathrm{d}x,
    \end{split}
\end{equation*}
since $f' = 0$ almost everywhere on $C(f)$.
Since both $G(f)$ and $D(f)$ are open, each is a countable union of open subintervals of $I$, which we denote by $\{I_k\}_{k\in \Xi}$ and $\{J_l\}_{l\in\Theta}$ respectively. On each $I_k$ $f$ is increasing and on each $J_l$, $f$ is decreasing. Hence, we get
\begin{equation*}
    \begin{split}
        V(f) &= \int_{G(f)}f'(x)\;\mathrm{d}x + \int_{D(f)}-f'(x)\;\mathrm{d}x\\
        &= \sum_{k\in \Xi}\int_{I_k}f'(x)\;\mathrm{d}x + \sum_{l\in\Theta}\int_{J_l}-f'(x)\;\mathrm{d}x\\
        &=\sum_{k\in \Xi}\int_{\mathbb{R}}\Big|\pi_0\big[(f|_{I_k})^{-1}(t)\big]\Big|\;\mathrm{d}t + \sum_{l\in\Theta}\int_{\mathbb{R}}\Big|\pi_0\big[(f|_{J_l})^{-1}(t)\big]\Big|\;\mathrm{d}t \\
        &=\int_{\mathbb{R}}\Big|\pi_0\big[(f|_{G(f)})^{-1}(t)\big]\Big|\;\mathrm{d}t + \int_{\mathbb{R}}\Big|\pi_0\big[(f|_{D(f)})^{-1}(t)\big]\Big|\;\mathrm{d}t \\
        &=\int_{\mathbb{R}}\Big|\pi_0\big[(f|_{G(f)\cup D(f)})^{-1}(t)\big]\Big|\;\mathrm{d}t.
    \end{split}
\end{equation*}
Here, the last line follows from the fact that if $x$ and $y$ have the same $f$ value, each point in $G(f)$ or $D(f)$, then by the definition of these sets there must be a point between them that obtains either a smaller or larger value of $f$. The line before follows from similar reasoning. If it is granted that $f(C(f))$ has measure zero we then have
\begin{equation*}
    \begin{split}
        V(f) &=\int_{\mathbb{R}}\Big|\pi_0\big[(f|_{G(f)\cup D(f)})^{-1}(t)\big]\Big|\;\mathrm{d}t\\
        &=\int_{\mathbb{R}}\Big|\pi_0\big[(f|_{G(f)\cup D(f)})^{-1}(t)\big]\Big|\;\mathrm{d}t + \int_{\mathbb{R}}\Big|\pi_0\big[(f|_{C(f)})^{-1}(t)\big]\Big|\;\mathrm{d}t\\
        &=\int_{\mathbb{R}}\Big|\pi_0\big[f^{-1}(t)\big]\Big|\;\mathrm{d}t\\
        &=\int_a^b\Big|\pi_0\big[f^{-1}(t)\big]\Big|\;\mathrm{d}t.
    \end{split}
\end{equation*}
Here, the second to last equality follows from the fact that if $x$ is in $C(f)$ and $y$ is in $G(f)$ or $D(f)$, then by definition of $G(f)$ and $D(f)$ there is a point between $x$ and $y$ with an $f$ value not equal to $f(y)$. In particular, the integrand at the end of this equation must be finite almost everywhere.

Now we show that indeed $f(C(f))$ has measure zero. Since $f$ is piece-wise $C^1$, let $T_1,\ldots,T_k$ be closed intervals covering $I$ on the interiors of which $f$ is $C^1$. Consider $Z_i$, the subset of the interior of $I_i$ with $f' = 0$. By Sard's Theorem (see for example \cite[Theorem 7.2]{sard1942measure}), $f(Z_i)$ has measure zero. We have that $C(f)$ is a subset of $Z_1\cup\ldots\cup Z_k\cup \partial T_1 \cup \ldots \cup \partial T_k$, whose image under $f$ has measure zero. Thus $f(C(f))$ has measure zero.

Hence, the desired result follows once we establish that for any $t\in[a,b]$,
\begin{equation*}
    \Big|\pi_0\big[f^{-1}(-\infty,t]\big]\Big| \leq \Big|\pi_0\big[f^{-1}(t)\big]\Big|.
\end{equation*}
This inequality implies that $\ECT_f(v,t)$ is defined for almost all $t$ since $\ECT_f(v,t)$ is just the left-hand side of this inequality, which is positive-valued, and bounded by a function that is finite for almost all $t$.

To prove this statement, it suffices to show that any path-component of $f^{-1}(-\infty,t]$ contains a point with $f$ value $t$. Suppose otherwise, that there is a path-component $C$ of $f^{-1}(-\infty,t]$ with $f(C)<t$. By continuity of $f$, $C$ must also be a path-component of $f^{-1}(-\infty,b]$. Indeed, suppose $\alpha$ is a path from the complement of $C$ to $C$. Thus, every neighborhood of $\alpha^{-1}(C)$ must intersect $(f\circ\alpha)^{-1}(t,b]$. But this produces a contradiction of the intermediate value theorem. So $C$ is a path-component of $I = f^{-1}(-\infty,b]$ and hence is $I$. The fact that $f(C)<t\leq b$ thus contradicts the definition of $b$ as the maximum of $f$, completing the proof.
\end{proof}

\begin{remark}
Suppose $\gamma: I \to \mathbb{R}^d$ is continuous and definable with respect to an o-minimal structure on $\mathbb{R}$.
By \cite[Chapter 7, Theorem 3.2]{van1998tame}), $\gamma$ is piece-wise $C^1$ and hence the above result applies.
\end{remark}

\begin{remark}
In the case where $f$ is tame, this result is implied by a stronger result of \cite[Corollary 4.6]{biswas2022window} for tame functions. The main contribution of this proposition is that the stated bound still holds when $f$ is not tame.
\end{remark}

With the previous result in mind, we establish a bound on the variation of a curve that is approximately straight.

\begin{lemma}
Suppose $\gamma:I \to \mathbb{R}^d$ is a piece-wise differentiable path with length $L$ and the first coordinate $\gamma_1$ satisfies $|\gamma_1(1)-\gamma_1(0)| = L_x$. Then the variation of any other coordinate function $\gamma_n$ of $\gamma$ is bounded by
\begin{equation}
    V(\gamma_n) \leq \sqrt{L^2 - L_x^2}.
\end{equation}
\label{lem:varbound}
\end{lemma}

\begin{proof}
Without loss of generality, we can assume that $\gamma(0) = 0$, $\gamma_1(1)= L_x$, and we can show this bound holds for $\gamma_2$ only.

From $\gamma$ we can construct another function $\bar{\gamma}$ by
\begin{equation*}
    \bar{\gamma}(t) := \int_0^t (\gamma'_1(t),\;|\gamma'_2(t)|, \;\gamma'_3(t),\ldots,\; \gamma'_d(t))\; \mathrm{d}t.
\end{equation*}
Put differently, $\bar{\gamma}$ has the same coordinate functions as $\gamma$ except in the second coordinate, where $\bar{\gamma}_2$ has the same absolute value of its derivative as $\gamma$, but is never decreasing. It is immediate that $\gamma$ and $\bar{\gamma}$ have the same length, value of $\gamma_1(1)$, and variation in the second coordinate. Hence, it suffices to show that the lemma holds for $\gamma_2$ for curves $\gamma$ with length $L$, $\gamma(0) = 0$, $\gamma_1(1) = L_x$, and $\gamma'_2(t)\geq0$ for all $t$.

The fact that $\gamma$ has length $L$ implies that $\gamma(1)$ lies in the closed $d$-disk of radius $L$ centred at the origin. The fact that $\gamma_1(x) = L_x$ implies that $\gamma(1)$ lies on the hyperplane of points with the first coordinate $L_x$. Elementary geometry shows that the intersection of the disk and the hyperplane is
\begin{equation*}
    \{(y_1,y_2,\ldots,y_n)\in \mathbb{R}^d : y_1 = L_x,\; y_2^2+\ldots +y_d^2 \leq L^2-L_x^2\}.
\end{equation*}
It is easily seen that the greatest value of $y_2$ on this set is $\sqrt{L^2-L_x^2}$. So $\gamma_2(1)$ is bounded above by this value. Hence,
\begin{equation*}
    V(\gamma_2) = \int_0^1 \big|\gamma_2'(t)\big|\;\mathrm{d}t = \int_0^1 \gamma_2'(t)\;\mathrm{d}t = \gamma_2(1)-\gamma_2(0) = \gamma_2(1) \leq \sqrt{L^2-L_x^2}.
\end{equation*}
\end{proof}

We can now bound the $L_1$ distance between Euler characteristic transforms of nearby curves, assuming one of them is approximately straight.
\begin{proposition}
Let $\alpha:I\to\mathbb{R}^d$ be a piece-wise $C^1$ map such that the distance of $\alpha(0)$ to $\alpha(1)$ is $L$, with arc length no greater than $L+\epsilon$. Let $\beta:I\to\mathbb{R}^d$ be another piece-wise $C^1$ map with arc length no greater than $L+2\epsilon$, and endpoints within $\epsilon$ of the corresponding endpoints of $\alpha$. Then
\begin{equation*}
    \big\|\ECT_\alpha-\ECT_\beta\big\| \leq \begin{cases}
            8\sqrt{L\epsilon} &  L > 2\epsilon \\
            10 \epsilon & L \leq 2\epsilon.
        \end{cases}
\end{equation*}
\label{prop:localstab}
\end{proposition}

\begin{proof}
Let $v$ be an arbitrary unit vector, $w = \alpha(1)-\alpha(0)$, and $\theta$ be the angle between $w$ and the hyperplane normal to $v$. After potentially applying a rotation to $\alpha$ and $\beta$, we may assume that $v = (0,1,0,\ldots,0)$.
By applying another rotation we may assume also that $w$ is only non-zero in the first two coordinates.

Let $f$ denote the inner-product with $v$ and let $a$ and $b$ be the minimum and maximum of $f\circ\alpha$ and $c$ and $d$ be the minimum and maximum of $f\circ\beta$. Throughout the proof, we use the fact that if both $\ECT_\alpha(v,t)$ and $\ECT_\beta(v,t)$ are non-zero, then
\begin{equation*}
    |\ECT_\alpha(v,t) - \ECT_\alpha(v,t)| \leq \ECT_\alpha(v,t) + \ECT_\beta(v,t) - 2,
\end{equation*}
as both Euler characteristic transforms must be positive (since subsets of the interval are component-wise contractible) and greater than 1.

First, suppose $\max(a,c) \leq \min(b,d)$. Then
\begin{equation*}
    \begin{split}
        \int_\mathbb{R}|\ECT_\alpha(v,t) - \ECT_\beta(v,t)|\; \mathrm{d}t &= \int_{\min(a,c)}^{\max(a,c)}|\ECT_\alpha(v,t) - \ECT_\beta(v,t)|\; \mathrm{d}t\\
        &+\int_{\max(a,c)}^{\min(b,d)}|\ECT_\alpha(v,t) - \ECT_\beta(v,t)|\; \mathrm{d}t\\
        &+\int_{\min(b,d)}^{\max(b,d)}|\ECT_\alpha(v,t) - \ECT_\beta(v,t)|\; \mathrm{d}t.\\
    \end{split}
\end{equation*}
Suppose also that $b\leq d$. Then the above expression is bounded by
\begin{equation*}
    \int_{\min(a,c)}^{\max(a,c)}\ECT_\alpha(v,t) + \ECT_\beta(v,t)\, \mathrm{d}t + \int_{\max(a,c)}^{b}\ECT_\alpha(v,t) + \ECT_\beta(v,t) - 2\, \mathrm{d}t + \int_{b}^{d} \ECT_\beta(v,t)-1\, \mathrm{d}t,
\end{equation*}
where we have only had to approximate the middle term. If we additionally suppose $a \leq c$, then our bound is equal to
\begin{equation*}
    \int_{a}^{c}\ECT_\alpha(v,t)\; \mathrm{d}t + \int_{c}^{b}\ECT_\alpha(v,t) + \ECT_\beta(v,t) - 2\; \mathrm{d}t + \int_{b}^{d} \ECT_\beta(v,t)-1\; \mathrm{d}t.
\end{equation*}

Rearranging and by the linearity of the integral, the above is equal to
\begin{equation*}
    \int_{c}^{d}\ECT_\beta(v,t)\; \mathrm{d}t + \int_{a}^{b}\ECT_\alpha(v,t)\mathrm{d}t - 2(b-c) - (d-b).
\end{equation*}
Similar analysis when $a>c$ and/or $b>d$ shows that when $\max(a,c) \leq \min(b,d)$,
\begin{equation}
    \begin{split}
        \int_\mathbb{R}|\ECT_\alpha(v,t) - \ECT_\beta(v,t)|\; \mathrm{d}t &\leq \int_{c}^{d}\ECT_\beta(v,t)\; \mathrm{d}t + \int_{a}^{b}\ECT_\alpha(v,t)\\ 
        &- 2|\min(b,d)-\max(a,c)| - |\max(b,d)-\min(b,d)|.
    \end{split}
    \label{eqn:ectdiff1}
\end{equation}

Note that $|b-a| = L|\sin\theta|$. By hypothesis $\| \alpha(0)-\beta(0)\|\leq \epsilon$, so $|a-c| \leq \epsilon$. Similarly $|b-d|\leq \epsilon$. Hence,
\begin{equation*}
    |\min(b,d)-\max(a,c)| \geq L|\sin\theta|-2\epsilon
\end{equation*}
by the triangle inequality. Of course, the quantity on the left is also positive, so
\begin{equation*}
    |\min(b,d)-\max(a,c)| \geq \max(0,L|\sin\theta|-2\epsilon).
\end{equation*}
Trivially, we also have $|\max(b,d)-\min(b,d)| \geq 0$. Applying these inequalities, Proposition \ref{prop:ectbound}, and Lemma \ref{lem:varbound} to Equation (\ref{eqn:ectdiff1}), we have
\begin{equation}
    \begin{split}
        \int_\mathbb{R}|\ECT_\alpha(v,t) - \ECT_\beta(v,t)|\; \mathrm{d}t &\leq \sqrt{(L+2\epsilon)^2-\max(0,L|\cos\theta|-2\epsilon)^2}\\
        &+\sqrt{(L+\epsilon)^2-L^2|\cos^2\theta|}\\
        &-2\max(0,L|\sin\theta|-2\epsilon).
    \end{split}
    \label{eqn:ectdiff2}
\end{equation}
In the application of Lemma \ref{lem:varbound} for the first term, we use that $|\beta_1(1)-\beta_1(0)|\leq \max(0,L|\cos\theta|-2\epsilon)$.

Otherwise, $\max(a,c) \geq \min(b,d)$, so either $a\leq b\leq c\leq d$ or $c\leq d\leq a\leq b$. In either of these cases, we must have that $L|\sin\theta|\leq \epsilon$. Consider the first of these cases.
We observe
\begin{equation*}
    \begin{split}
        \int_\mathbb{R}|\ECT_\alpha(v,t) - \ECT_\beta(v,t)|\; \mathrm{d}t &= \int_{a}^{b}|\ECT_\alpha(v,t) - \ECT_\beta(v,t)|\; \mathrm{d}t\\
        &+\int_{b}^{c}|\ECT_\alpha(v,t) - \ECT_\beta(v,t)|\; \mathrm{d}t\\
        &+\int_{c}^{d}|\ECT_\alpha(v,t) - \ECT_\beta(v,t)|\; \mathrm{d}t.\\
    \end{split}
\end{equation*}
This quantity is equal to
\begin{equation*}
    \int_{a}^{b}\ECT_\alpha(v,t)\; \mathrm{d}t + \int_{b}^{c}1\; \mathrm{d}t +\int_{c}^{d}\ECT_\beta(v,t) - 1\; \mathrm{d}t,
\end{equation*}
Bounding from above, we have
\begin{equation*}
    \begin{split}
        \int_\mathbb{R}|\ECT_\alpha(v,t) - \ECT_\beta(v,t)|\; \mathrm{d}t &\leq \int_{a}^{b}\ECT_\alpha(v,t)\; \mathrm{d}t +\int_{c}^{d}\ECT_\beta(v,t)\; \mathrm{d}t  + (c-b) - (d-c)\\
        &\leq\int_{a}^{b}\ECT_\alpha(v,t)\; \mathrm{d}t +\int_{c}^{d}\ECT_\beta(v,t)\; \mathrm{d}t  + (c-b).
    \end{split}
\end{equation*}
Similar analysis when $c\leq d\leq a\leq b$ shows that in general, if $\max(a,c) \leq \min(b,d)$, then
\begin{equation*}
    \int_\mathbb{R}|\ECT_\alpha(v,t) - \ECT_\beta(v,t)|\; \mathrm{d}t \leq \int_{a}^{b}\ECT_\alpha(v,t)\; \mathrm{d}t + \int_{c}^{d}\ECT_\beta(v,t)\; \mathrm{d}t + \min(|c-b|,|d-a|).
\end{equation*}
By the triangle inequality, $\min(|c-b|,|d-a|) \leq \epsilon - L\sin\theta$. Applying Proposition \ref{prop:ectbound} and Lemma \ref{lem:varbound} once again, we see
\begin{equation*}
    \begin{split}
        \int_\mathbb{R}|\ECT_\alpha(v,t) - \ECT_\beta(v,t)|\; \mathrm{d}t&\leq \sqrt{(L+2\epsilon)^2-\max(0,L|\cos\theta|-2\epsilon)^2}\\
        &+ \sqrt{(L+\epsilon)^2-L^2|\cos^2\theta|}\\ 
        &+ \max(0,\epsilon-L|\sin\theta|).
    \end{split}
\end{equation*}

In summary, $\int_\mathbb{R}|\ECT_\alpha(v,t) - \ECT_\beta(v,t)|\; \mathrm{d}t$ is bounded by
\begin{equation*}
    \sqrt{(L+2\epsilon)^2-\max(0,L|\cos\theta|-2\epsilon)^2}+\sqrt{(L+\epsilon)^2-L^2|\cos^2\theta|} - 2\max(0,L|\sin\theta|-2\epsilon)
\end{equation*}
whenever $L|\sin\theta|\geq \epsilon$. Otherwise, either we still have $\max(a,c) \leq \min(b,d)$ and the above bound still holds or $\max(a,c) \geq \min(b,d)$ and we instead have the bound
\begin{equation*}
    \sqrt{(L+2\epsilon)^2-\max(0,L|\cos\theta|-2\epsilon)^2}+\sqrt{(L+\epsilon)^2-L^2|\cos^2\theta|} + \max(0,\epsilon-L|\sin\theta|).
\end{equation*}

Hence, in general, 
\begin{equation*}
    \begin{split}
        \int_\mathbb{R}|\ECT_\alpha(v,t) - \ECT_\beta(v,t)|\; \mathrm{d}t &= \sqrt{(L+2\epsilon)^2-\max(0,L|\cos\theta|-2\epsilon)^2}\\
        &+ \sqrt{(L+\epsilon)^2-L^2|\cos^2\theta|}\\ 
        &- 2\max(0,L|\sin\theta|-2\epsilon) + \max(0,\epsilon-L|\sin\theta|).
    \end{split}
\end{equation*}
The proof is complete once we have established the following tedious lemma.
\end{proof}

\begin{lemma}
The function
\begin{equation*}
    \begin{split}
        f(\theta) &= \sqrt{(L+2\epsilon)^2-\max(0,L|\cos\theta|-2\epsilon)^2} + \sqrt{(L+\epsilon)^2-L^2|\cos^2\theta|}\\ 
        &- 2\max(0,L|\sin\theta|-2\epsilon) + \max(0,\epsilon-L|\sin\theta|)
    \end{split}
\end{equation*}
is bounded above by
\begin{equation*}
    f(\theta) \leq \begin{cases}
            8\sqrt{L\epsilon} &  L > 2\epsilon \\
            10 \epsilon & L \leq 2\epsilon. 
        \end{cases}
\end{equation*}
\end{lemma}

\begin{proof}
Thanks to the symmetries of the sine and cosine functions and the absolute values present in the formula for $f$, we have that $f(-\theta) = f(\theta)$ and $f(\pi/2 - \theta) = f(\theta)$. So $f(\theta) = f(\pi/2 + \theta)$. Therefore it suffices to bound $f$ on the interval $[0,\pi/2]$. On this interval, we can remove the absolute values in the formula for $f$, giving 
\begin{equation*}
    \begin{split}
        f(\theta) &= \sqrt{(L+2\epsilon)^2-\max(0,L\cos\theta-2\epsilon)^2} + \sqrt{(L+\epsilon)^2-L^2\cos^2\theta}\\ 
        &- 2\max(0,L\sin\theta-2\epsilon) + \max(0,\epsilon-L\sin\theta).
    \end{split}
\end{equation*}
With the exception of finitely many values of $\theta$, the derivative of $f$ exists and is equal to
\begin{equation*}
    \begin{split}
        \mathbb{I}(L\cos\theta > 2\epsilon)\frac{L\sin\theta(L\cos\theta-2\epsilon)}{\sqrt{(L+2\epsilon)^2-(L\cos\theta-2\epsilon)^2}} &+ \frac{L^2\sin\theta\cos\theta}{\sqrt{(L+\epsilon)^2-L^2\cos^2\theta}}\\ - \mathbb{I}(L\sin\theta>2\epsilon)2L\cos(\theta) - \mathbb{I}(L\sin\theta&<\epsilon)L\cos\theta,
    \end{split}
\end{equation*}
where $\mathbb{I}$ denotes the indicator function.

Notice that
\begin{equation*}
    \frac{L^2\sin\theta\cos\theta}{\sqrt{(L+\epsilon)^2-L^2\cos^2\theta}} \leq \frac{L^2\sin\theta\cos\theta}{\sqrt{L^2-L^2\cos^2\theta}} = L\cos\theta,
\end{equation*}
and similarly
\begin{equation*}
    \frac{L\sin\theta(L\cos\theta-2\epsilon)}{\sqrt{(L+2\epsilon)^2-(L\cos\theta-2\epsilon)^2}} \leq \frac{L^2\sin\theta\cos\theta}{\sqrt{L^2-L^2\cos^2\theta}} = L\cos\theta.
\end{equation*}
Using these identities, we get that
\begin{equation*}
    f'(\theta) \leq \mathbb{I}(L\cos\theta>2\epsilon)L\cos\theta + L\cos\theta - \mathbb{I}(L\sin\theta>2\epsilon)2L\cos(\theta) - \mathbb{I}(L\sin\theta<\epsilon)L\cos\theta.
\end{equation*}
Hence $f$ is weakly decreasing whenever $L\sin\theta>2\epsilon$. When $\epsilon < L\sin\theta<2\epsilon$, every non-zero term in $f'$ is positive, and so $f$ is increasing. We now bound $f'$ in absolute value when $L\sin\theta< \epsilon$. In this case,
\begin{equation*}
    \begin{split}
        |f'(\theta)| &\leq \mathbb{I}(L\cos\theta > 2\epsilon)\frac{L\sin\theta(L\cos\theta-2\epsilon)}{\sqrt{(L+2\epsilon)^2-(L\cos\theta-2\epsilon)^2}} + \frac{L^2\sin\theta\cos\theta}{\sqrt{(L+\epsilon)^2-L^2\cos^2\theta}} + L\cos\theta \\
        & \leq 3L\cos\theta \\
        &\leq 3L,
    \end{split}
\end{equation*}
using our approximations for the first and second terms from earlier.

Further, if $L > 2\epsilon$,
\begin{equation*}
    \begin{split}
    f(0) &= \sqrt{(L+2\epsilon)^2-\max(0,L-2\epsilon)^2} + \sqrt{(L+\epsilon)^2-L^2} - \epsilon \\
    & = \sqrt{8L\epsilon} + \sqrt{2L\epsilon + \epsilon^2} - \epsilon\\
    &\leq (\sqrt{8} + \sqrt{3})\sqrt{L\epsilon} - \epsilon.
    \end{split}
\end{equation*}
Otherwise $L\leq 2\epsilon$ and,
\begin{equation*}
    \begin{split}
    f(0) &= \sqrt{(L+2\epsilon)^2-\max(0,L-2\epsilon)^2} + \sqrt{(L+\epsilon)^2-L^2} - \epsilon \\
    & = L + 2\epsilon + \sqrt{2L\epsilon + \epsilon^2} - \epsilon \\
    & = L + \epsilon + \sqrt{2L\epsilon + \epsilon^2}\\
    & \leq (3+\sqrt{5}) \epsilon.
    \end{split}
\end{equation*}

Hence, we can bound $f(\theta)$ for $\theta$ on the interval $[0, \sin^{-1}(\epsilon/L)]$ (or $[0,\pi/2]$ if $\epsilon > L$) by using our upper bounds for $f(0)$ and $|f'(\theta)|$ on this interval. By additionally using the inequality $\sin^{-1}(x)\leq \pi x/2$ for positive $x$, we obtain that when $L\sin\theta \leq \epsilon$,
\begin{equation*}
    \begin{split}
        f(\theta) &\leq \begin{cases} 
            (\sqrt{8} + \sqrt{3})\sqrt{L\epsilon} + (3\pi/2-1)\epsilon &  L >   2\epsilon \\
            (3+\sqrt{5} + 3\pi/2) \epsilon & \epsilon \leq L \leq 2\epsilon \\
            (3+\sqrt{5}) \epsilon + (3\pi/2)L & L < \epsilon 
        \end{cases}\\
        &\leq \begin{cases}
            (\sqrt{8} + \sqrt{3})\sqrt{L\epsilon} + (3\pi/2-1)\epsilon &  L >  2\epsilon \\
            (3+\sqrt{5} + 3\pi/2) \epsilon & L \leq 2\epsilon 
        \end{cases}\\
        & \leq \begin{cases}
            (\sqrt{8} + \sqrt{3} + (3\pi/4-1/2)\sqrt{2})\sqrt{L\epsilon} &  L > 2\epsilon \\
            (3+\sqrt{5} + 3\pi/2) \epsilon & L \leq 2\epsilon 
        \end{cases}\\
        & \leq \begin{cases}
            8\sqrt{L\epsilon} &  L > 2\epsilon \\
            10 \epsilon & L \leq 2\epsilon.
        \end{cases}
    \end{split}
\end{equation*}
Otherwise, we know $f$ is weakly increasing until $L\sin\theta > 2\epsilon$, after which point it is weakly decreasing. Hence, if $f$ is not maximised where $L\sin\theta \leq \epsilon$, it must attain its maximum when $L\sin\theta = 2\epsilon$, or equivalently $\theta = \sin^{-1}(2\epsilon/L)$. Note that this implies $2\epsilon \leq L$. We compute
\begin{equation*}
    \begin{split}
        f(\sin^{-1}(2\epsilon/L)) &= \sqrt{(L+2\epsilon)^2-\max(0,\sqrt{L^2-4\epsilon^2}-2\epsilon)^2} + \sqrt{(L+\epsilon)^2-L^2-4\epsilon^2}\\
        & = \sqrt{(L+2\epsilon)^2-\max(0,\sqrt{L^2-4\epsilon^2}-2\epsilon)^2} + \sqrt{2L\epsilon - 3\epsilon^2}\\
        & =\begin{cases}
            \sqrt{4L\epsilon + 4\epsilon^2 + 2\epsilon\sqrt{L^2-4\epsilon^2}} + \sqrt{2L\epsilon - 3\epsilon^2}  &  L > 2\sqrt{2}\epsilon \\
            L + 2\epsilon + \sqrt{2L\epsilon - 3\epsilon^2} & 2\epsilon \leq L \leq 2\sqrt{2}\epsilon 
        \end{cases}\\
        & \leq\begin{cases}
            \sqrt{6L\epsilon + 4\epsilon^2} + \sqrt{2L\epsilon}  &  L > 2\sqrt{2}\epsilon \\
            (2 + 2\sqrt{2})\epsilon + \sqrt{2L\epsilon} & 2\epsilon \leq L \leq 2\sqrt{2}\epsilon 
        \end{cases}\\
        & \leq\begin{cases}
            (\sqrt{6+\sqrt{2}} +\sqrt{2})\sqrt{L\epsilon}  &  L > 2\sqrt{2}\epsilon \\
            (2 + 2\sqrt{2})\sqrt{L\epsilon} & 2\epsilon \leq L \leq 2\sqrt{2}\epsilon 
        \end{cases}\\
        &\leq 5\sqrt{L\epsilon}.
    \end{split}
\end{equation*}
Hence, to totally bound $f(\theta)$ on the interval $[0,\pi/2]$ we need only use our earlier bound, namely
\begin{equation*}
    f(\theta) \leq \begin{cases}
            8\sqrt{L\epsilon} &  L > 2\epsilon \\
            10 \epsilon & L \leq 2\epsilon.
        \end{cases}
\end{equation*}
\end{proof}

The goal of the following proposition is to bound from below the chord length of a short segment of a curve given that it has bounded curvature.

\begin{proposition}
Suppose $\gamma:[0,L]\to\mathbb{R}^d$ is a twice differentiable curve parametrised by arc length with curvature $\kappa$ bounded in norm by $M$. Let $0<\epsilon < \pi/M$. Then for any $t\in [0,L - \epsilon]$,
\begin{equation*}
    \epsilon \geq \|\gamma(t+\epsilon)-\gamma(t)\|_2\geq \frac{2}{M}\sin\left(\frac{M}{2}\epsilon\right).
\end{equation*}
In particular,
\begin{equation*}
\|\gamma(t+\epsilon)-\gamma(t)\|_2\geq \epsilon - \frac{M^2}{24}\epsilon^3.
\end{equation*}
\label{prop:curvbound}
\end{proposition}

To prove this we make use of the following theorem of Schwarz, which we cite from \cite{chern1967curves}:
\begin{theorem}[Schwarz]
Let $C$ be an arc joining two given points $A$ and $B$ with curvature $\kappa(s) \leq 1/R$, such that $R\geq \frac{1}{2}\delta$, where $\delta$ is the distance between $A$ and $B$. Let $S$ be a circle of radius $R$ through $A$ and $B$. Then the length of $C$ is either less than, or equal to, the shorter arc $AB$ or greater than, or equal to, the longer arc $AB$ on $S$.
\end{theorem}

\begin{proof}[Proof of Proposition \ref{prop:curvbound}]
The first inequality is clear since $\gamma$ is parameterised by arc length. Now fix $t$. For the second inequality, consider the optimisation problem of minimising $\|\alpha(t+\epsilon)-\alpha(t)\|_2$ subject to the constraints that $\|\alpha'\|_2 = 1$ and $\|\alpha''\|_2 \leq M$. Consider an arc of length $\epsilon$ on the circle of curvature $M$, which has radius $1/M$. Elementary geometry shows that the distance between the endpoints of such an arc is $\frac{2}{M}\sin(M\epsilon/2)$. We claim that this arc provides an optimal solution. Indeed, let $\gamma$ be any curve that performs at least as well as this arc in the sense that
\begin{equation*}
\|\gamma(t+\epsilon)-\gamma(t)\|_2 \leq \frac{2}{M}\sin\Big{(}\frac{M}{2}\epsilon\Big{)},
\end{equation*}
while $\|\gamma'\|_2 = 1$, $\|\gamma''\|_2 \leq M$.

Let $S$ be a circle of radius $1/M$ crossing both $\gamma(t)$ and $\gamma(t+\epsilon)$. Such a circle must exist since $\|\gamma(t+\epsilon)-\gamma(t)\|_2\leq 2/M$. The curve $\gamma$ on the interval $[t,t+\epsilon]$ is of length $\epsilon < \pi/M$ while the longer arc on $S$ connecting $\gamma(t)$ and $\gamma(t+\epsilon)$ has length greater than $\pi/M$. Hence, by the theorem of Schwarz, $\epsilon$ is less than or equal to the length of the shorter arc on $S$ from $\gamma(t)$ to $\gamma(t+\epsilon)$. If this inequality is strict, we may take a shorter portion of the circular arc with length $\epsilon$, which has a shorter distance between endpoints. This proves that an arc of length $\epsilon$ on a circle of radius $1/M$ is an optimal solution of the optimisation problem.

Thus for potentially suboptimal $\gamma$, we have
\begin{equation*}
\|\gamma(t+\epsilon)-\gamma(t)\|_2\geq \frac{2}{M}\sin\left(\frac{M}{2}\epsilon\right).
\end{equation*}
For the last statement of the proposition, by the Lagrange remainder theorem
\begin{equation*}
    \sin x -x + x^3/6 = \int_0^x \cos t \frac{(x-t)^5}{5!}\;\mathrm{d}t.
\end{equation*}
The right side is clearly positive provided that $0<x\leq \pi/2$. Since $0<\frac{M}{2}\epsilon < \pi/2$, we have
\begin{equation*}
    \frac{2}{M}\sin\left(\frac{M}{2}\epsilon\right) \geq \frac{2}{M}\left[\frac{M}{2}\epsilon - \frac{M^3}{48}\epsilon^3\right] = \epsilon - \frac{M^2}{24}\epsilon^3.
\end{equation*}
\end{proof}

We now use the results we have already proven about curves that are approximately straight to obtain a stability result for the Euler characteristic transform of more general shapes. To do this, we prove a lemma that allows us to glue together Euler characteristic transforms of functions restricted to different regions of a domain.

\begin{definition}
Let $V^* = (V, V_0,\{\Phi_\lambda\}_{\lambda\in\Lambda_V})$ and $W^* = (W, W_0,\{\Phi_\lambda\}_{\lambda\in\Lambda_W})$ be finite one-dimensional CW complexes, each with a fixed CW structure. Suppose there exist maps $f_V\in\mathcal{F}^r(V^*,d)$ and $f_W\in\mathcal{F}^r(W^*,d)$, a subset $S\subseteq V_0$ and an injective map $m:S \to W_0$ such that $f_V = f_W \circ m$ on $S$. We define the glue of $V^*$ and $W^*$ under $m$ to be a finite complex with structure:
\begin{equation*}
Z^* = (Z, Z_0, \{\Phi_\lambda\}_{\lambda\in\Lambda_Z}) : = \big((V \sqcup W)/m,\;\; (V_0 \sqcup W_0)/m,\;\; \{\Phi_\lambda\}_{\lambda\in\Lambda_V\sqcup\Lambda_W}\big).
\end{equation*}
We define the glue of $f_V$ and $f_W$ under $m$ to be the map $f_Z:Z \to \mathbb{R}^d$ which restricts to $f_V$ on $V$ and $f_W$ of $W$. This map is well defined (since $f_V = f_W \circ m$ on $S$) and is an element of $\funcspacer$.
\end{definition}

\begin{lemma}
Using the notation of the previous definition, suppose $\ECT_{f_V}(v,t)$ and $\ECT_{f_W}(v,t)$ are defined for almost all $t$ for any fixed $v$. Then
\begin{equation}
    \ECT_{f_Z}(v,t) =  \ECT_{f_V}(v,t) +  \ECT_{f_W}(v,t) -  \ECT_{f_S}(v,t)
    \label{eqn:inc-exc}
\end{equation}
for almost all $t$ when $v$ is fixed, where $f_S$ is the restriction of $f_V$ to $S$.
\label{lem:glue}
\end{lemma}

\begin{proof}
Fix a unit vector $v$ in $\mathbb{R}^d$. Let $p_1,\ldots, p_k$ be the points of $S$. We denote by $V(v,t)$ the subset of points $x$ in $V$ satisfying that $\langle v,f_V(x)\rangle \leq t$. We define $W(v,t)$ and $Z(v,t)$ analogously. We let $S(v,t)$ denote the intersection of $S$ and $V(v,t)$.

Via the inclusions of $V$ and $W$ into $Z$, we can view $Z$ as the union of $V$ and $W$, with $V$ and $W$ intersecting in $Z$ at $S$. Similarly, we can view $Z(v,t)$ as the union of $V(v,t)$ and $W(v,t)$, with these two subsets intersecting at $S(v,t)$.  For almost all $t\in\mathbb{R}$, $v\cdot f_V(p_i) \neq t$ for all $i$. Fix any such $t$. Hence, we have that the interiors of $V(v,t)$ and $W(v,t)$ cover their intersection $S(v,t)$, by continuity of $f_V$ and $f_W$. Therefore, we have a Mayer-Vietoris exact sequence of homology groups \cite[149]{hatcher2002algebraic}:
\begin{equation*}
    \ldots \to H_i\big(S(v,t)\big) \to H_i\big(V(v,t)\big) \oplus H_i\big(W(v,t)\big) \to H_i\big(Z(v,t)\big) \to \ldots.
\end{equation*}

A routine argument then deduces the identity
\begin{equation*}
    \chi\big(Z(v,t)\big) = \chi\big(V(v,t)\big) +\chi\big(W(v,t)\big) - \chi\big(S(v,t)\big),
\end{equation*}
whenever all Euler characteristics on the right-hand side are defined. This happens for almost all $t$ and is another way of writing the identity of Equation (\ref{eqn:inc-exc}).
\end{proof}

We now have the prerequisites to prove Proposition \ref{prop:funcstab}.

\begin{proof}[Proof of Proposition \ref{prop:funcstab}]
Let $\alpha_\lambda := f \circ \Phi_\lambda$ and $\beta_\lambda := g \circ \Phi_\lambda$. Since the index set $\Lambda$ is finite, we let $\Lambda = \{1,\ldots,k\}$. We define $Z^0:=Z_0$, and inductively, $Z^\lambda = Z^{\lambda-1} \cup \im \Phi_\lambda$ for $\lambda \in \Lambda$. We then let $f_\lambda = f|_{Z^\lambda}$ and $g_\lambda = g|_{Z^\lambda}$.

Inductively, we assume that 
\begin{equation*}
\big\|\ECT_{f_{\lambda - 1}}-\ECT_{g_{\lambda - 1}}\| \leq |Z_0|\epsilon + \sum_{k = 1}^{\lambda-1} G_k(\epsilon).
\end{equation*}
Indeed, as a base case, it is easily observed that
\begin{equation*}
\big\|\ECT_{f_0}-\ECT_{g_0}\big\| \leq |Z_0|\epsilon.
\end{equation*}

We can split $\alpha_\lambda$ into $n$ pieces by restricting $\alpha_{\lambda,i}:[\frac{i-1}{n},\frac{i}{n}]\to \mathbb{R}^d$. Analogously we can split $\beta_\lambda$ into curves $\beta_{\lambda,i}$. By Proposition \ref{prop:curvbound}, the arc length of each $\alpha_{\lambda,i}$ is at most $M^2L_\lambda^3/24n^3$ greater than the distance between its endpoints, if $n > L_\lambda M/\pi$. We now apply Proposition \ref{prop:localstab}. Thus, provided
\begin{equation*}
    \frac{M^2L_\lambda^3}{24n^3} \leq \epsilon, \text{ or equivalently, } n \geq \Big{(}\frac{M^2L_\lambda^3}{24\epsilon}\Big{)}^{1/3},
\end{equation*}
we observe
\begin{equation*}
    \big\|\ECT_{\alpha_{\lambda,i}}-\ECT_{\beta_{\lambda,i}}\big\| \leq \begin{cases}
            8\sqrt{L_\lambda\epsilon/n} &  L_\lambda/n > 2\epsilon \\
            10 \epsilon & L_\lambda/n \leq 2\epsilon.
        \end{cases}
\end{equation*}
Let $m_\lambda\in\{1,2\}$ be the number of 0-cells (i.e. elements of $Z_0$) in the image of $\Phi_\lambda$. By repeatedly applying Lemma \ref{lem:glue} we have that
\begin{equation*}
    \ECT_{\alpha_\lambda}(v,t) = (m_\lambda-2)\ECT_{\alpha_\lambda(0)}(v,t) +\sum_{i=1}^n\ECT_{\alpha_{\lambda,i}}(v,t) - \sum_{i=1}^{n-1} \ECT_{\alpha_{\lambda,i}(i/n)}(v,t),
\end{equation*}
for almost all $t$ when $v$ is fixed. By the same argument, a similar equality holds for $\ECT_{\beta_\lambda}$. Hence, by the triangle inequality, we deduce that $\|\ECT_{\alpha_\lambda}-\ECT_{\beta_\lambda}\|$ is bounded above by
\begin{equation*}
    \begin{split}
        (m_\lambda-2)\big\|\ECT_{\alpha_\lambda(0)} - \ECT_{\beta_\lambda(0)}\big\| + \sum_{i=1}^n &\big\|\ECT_{\alpha_{\lambda,i}} - \ECT_{\beta_{\lambda,i}}\big\| + \sum_{i=1}^{n-1}\big\|\ECT_{\alpha_{\lambda,i}(i/n)} -\ECT_{\beta_{\lambda,i}(i/n)}\big\|\\
        &\leq \begin{cases}
                8\sqrt{L_\lambda n\epsilon} + (n + m_\lambda -3)\epsilon &  L_\lambda/n > 2\epsilon \\
                (11n +m_\lambda - 3)\epsilon & L_\lambda/n \leq 2\epsilon.
            \end{cases}
    \end{split}
\end{equation*}
In particular, this bound hold when we let 
\begin{equation*}
    n = n_\lambda := \max\left(\left\lceil \left(\frac{M^2L_\lambda^3}{24\epsilon}\right)^{1/3} \right\rceil,\left\lceil \frac{L_\lambda M}{\pi}\right\rceil\right).
\end{equation*}

Applying Lemma \ref{lem:glue} again, we have
\begin{equation*}
\ECT_{f_\lambda}(v,t) = \ECT_{f_{\lambda-1}}(v,t) + \ECT_{\alpha_\lambda}(v,t) - \ECT_{\alpha_\lambda(0)}(v,t) - (m_\lambda- 1)\ECT_{\alpha_\lambda(1)}(v,t),
\end{equation*}
for almost all $t$ when $v$ is fixed. Here, $\delta_{\lambda,0}$ is 1 if $\Phi_\lambda(0)\in S$ and zero otherwise. The value $\delta_{\lambda,1}$ is defined analogously for $\Phi_\lambda(1)$. Similarly, such an equation holds involving $g_\lambda$, $g_{\lambda-1}$, and $\beta_\lambda$.

Applying the triangle inequality as before, along with our bound for $\|\ECT_{\alpha_\lambda} - \ECT_{\beta_\lambda}\|$, we deduce
\begin{equation*}
    \big\|\ECT_{f_\lambda} - \ECT_{g_\lambda}\big\| \leq \big\|\ECT_{f_{\lambda-1}} - \ECT_{g_{\lambda-1}}\big\| + \big\|\ECT_{\alpha_\lambda} - \ECT_{\beta_\lambda}\big\| + (2-m_\lambda)\epsilon.
\end{equation*}
The last two terms sum to $G_\lambda(\epsilon) - \epsilon$, so in particular we have the bound
\begin{equation*}
    \big\|\ECT_{f_\lambda} - \ECT_{g_\lambda}\big\| \leq |Z_0|\epsilon + \sum_{k=1}^\lambda G_k(\epsilon).
\end{equation*}
Induction then proves the proposition.
\end{proof}

From Proposition \ref{prop:funcstab}, Theorem \ref{thm:ECTstab} follows easily.

\begin{proof}[Proof of Theorem \ref{thm:ECTstab}]
Fix some $X,Y\in\imspacer$ and suppose $d_{Z^*}(X,Y) < \epsilon$. Hence, we may choose $h_X,h_Y\in\homspacer$ with the properties given in Definition \ref{def:distdef}. Suppose that $X$ has curvature bounded by $M$ under $Z^*$. It follows that $h_X$ also has curvature bounded by $M$. Proposition \ref{prop:funcstab} gives that
\begin{equation*}
\left\|\ECT_{h_X}-\ECT_{h_Y}\right\| \leq |Z_0|\epsilon + \sum_{\lambda\in\Lambda} G_\lambda(\epsilon),
\end{equation*}
but $\ECT_X = \ECT_{h_X}$ and $\ECT_Y = \ECT_{h_Y}$ since $h_X$ and $h_Y$ are homeomorphisms. The second statement of the theorem follows.

For the first statement, note that every $X\in \imspacer$ has a bound $M$ on its curvature and that $G_\lambda(\epsilon) \to 0$ as $\epsilon \to 0$ for all $\lambda \in \Lambda$.
\end{proof}

\subsection{Stability of Piecewise Linear Interpolation}

If we are given $X\subseteq\mathbb{R}^d$, the $C^2$-image of a homeomorphism $h$ from some one-dimensional CW complex $Z$, it may not be easy to exactly compute $\ECT_X$. The main goal of this section is to show that a dense subset of $Z$ can be used to approximate $\ECT_h = \ECT_X$. First, we make precise the kind of dense subset we need to properly estimate the $\ECT_h$.
\begin{definition}\label{def:dense}
Let $Z^* = (Z, Z_0,\{\Phi_\lambda\}_{\lambda\in\Lambda})$ be a connected finite one-dimensional CW complex with some fixed cellular decomposition and $f$ be a $C^2$ map $f:Z \to \mathbb{R}^d$. We say that $A = \{a_1,\ldots,a_n\}\subseteq Z$ is a \emph{compatible subset} of $Z^*$ if the following hold:
\begin{enumerate}
    \item $Z_0 \subseteq A$ and
    \item $A-Z_0$ contains a point in each 1-cell of $Z$.
\end{enumerate}
These requirements ensure that $Z - A$ is a union of disjoint open intervals. If additionally, the length of the image of each of  these intervals under $f$ is less than $\epsilon$, we say that $A$ is an \emph{$\epsilon$-dense subset} for $f$. An infinite subset of $Z$ is \emph{compatible and dense} for $f$ if it contains an $\epsilon$-dense subset for all positive $\epsilon$.
\end{definition}

\begin{definition}\label{def:lin_interpolation}
Let $f$ be as in the previous definition and $A=\{a_1,\ldots,a_n\}$ be a compatible subset of $Z^*$. Let $a_{ij}$ denote the line segment from $a_i$ to $a_j$. We define a multiset $E$ with elements in the set of unordered pairs in $\{1,\ldots,n\}$. $E$ contains a copy of $(i,j)$ for each open curve in $Z-A$ whose endpoints are $a_i$ and $a_j$. We define

\begin{equation*}
    \ECT^A_f(v,t) = \#\{1\leq i\leq n: \langle f(a_i),v\rangle \leq t\}- \#\{(i,j)\in E: \max(\langle f(a_i),v\rangle, \langle f(a_j), v\rangle) \leq t\}.
\end{equation*}
\end{definition}

The main theorem of the section says that we can use dense subsets to approximate the Euler characteristic transform of a one-dimensional CW complex:

\begin{theorem}
\label{thm:interpolationstable}
Let $Z^* = (Z, Z_0,\{\Phi_\lambda\}_{\lambda\in\Lambda})$ be a connected finite one-dimensional CW complex with some fixed cellular decomposition and $f$ be a $C^2$ map $f:Z \to X \subseteq \mathbb{R}^d$. Suppose that $f$ has curvature bounded by $M$ and let $A$ be an $\epsilon$-dense subset of $Z$, where $0 < \epsilon < \pi/M$.  Let $L$ be the sum of the arc lengths of the images of 1-cells of $Z$ under $f$. Then
\begin{equation*}
    \left\|\ECT_f-\ECT^A_f\right\| \leq \frac{1}{\sqrt{12}}ML\epsilon.
\end{equation*}
\end{theorem}

In practice, the Theorem \ref{thm:interpolationstable} implies that the ECT of a function on a one-dimensional CW complex can be computed approximately via a dense subset. The proof of this theorem is similar to the proof of Theorem \ref{thm:ECTstab}, but requires two additional lemmas.

\begin{lemma}
\label{lem:ECTapproxeqn}
Using the notation Definition \ref{def:lin_interpolation}, let $b_i:=f(a_i)$, and $b_{ij}$ denote the line segment from $b_i$ to $b_j$, and let $c_i$ denote the number of pairs in $E$ containing $i$. Then 
\begin{equation*}
    \ECT^A_f = \sum_{(i,j) \in E} \ECT_{b_{ij}} - \sum_{i = 1}^n (c_i-1)\ECT_{b_i}.
\end{equation*}
\end{lemma}

\begin{proof}
Fix some $v$ and $t$. If $\max(\langle b_i,v\rangle, \langle b_j, v\rangle) \leq t$, then $\ECT_{b_{ij}}(v,t) = \chi(b_{ij}) = 1$. If instead $\min(\langle b_i,v\rangle, \langle b_j, v\rangle)> t$, then $\ECT_{b_{ij}}(v,t) = \chi(\emptyset) = 0$. Otherwise, without loss of generality, suppose $\langle b_i,v \rangle \leq t$ and $\langle b_j,v \rangle > t$. Again, we have that $\ECT_{b_{ij}(v,t)}$ is equal to the Euler characteristic of a line segment, which is equal to 1.

Define the submultisets
\begin{equation*}
    \begin{split}
        E_{\textrm{up}}& = \{(i,j)\in E: \min(\langle b_i,v\rangle, \langle b_j, v\rangle) > t\},\\
        E_{\textrm{down}}&=\{(i,j)\in E: \max(\langle b_i,v\rangle, \langle b_j, v\rangle) \leq t\},\\
        E_{\textrm{mid}}&= \{(i,j)\in E: \max(\langle b_i,v\rangle, \langle b_j, v\rangle) > t, \; \min(\langle b_i,v\rangle, \langle b_j, v\rangle) \leq t\}.
    \end{split}
\end{equation*}
Note $E = E_{\textrm{up}} \sqcup E_{\textrm{down}} \sqcup E_{\textrm{mid}}$. Therefore,
\begin{equation*}
    \begin{split}
        &\sum_{(i,j) \in E} \ECT_{b_{ij}}(v,t) - \sum_{i = 1}^n (c_i-1)\ECT_{b_i}(v,t)\\
        &=\sum_{(i,j) \in E_{\textrm{up}}} \ECT_{b_{ij}}(v,t) + \sum_{(i,j) \in E_{\textrm{down}}} \ECT_{b_{ij}}(v,t)\\
        &\qquad\qquad+ \sum_{(i,j) \in E_{\textrm{mid}}} \ECT_{b_{ij}}(v,t) - \sum_{i = 1}^n (c_i-1)\ECT_{b_i}(v,t)\\
        & = \sum_{(i,j) \in E_{\textrm{down}}} 1 + \sum_{(i,j) \in E_{\textrm{mid}}} 1 - \sum_{i = 1}^n (c_i-1)\ECT_{b_i}(v,t)\\
        &= \sum_{(i,j) \in E_{\textrm{down}}} (2-1) + \sum_{(i,j) \in E_{\textrm{mid}}} 1 - \sum_{i = 1}^n (c_i-1)\ECT_{b_i}(v,t)\\
        &= \sum_{(i,j) \in E_{\textrm{down}}} 2 + \sum_{(i,j) \in E_{\textrm{mid}}} 1 - \sum_{i = 1}^n (c_i-1)\ECT_{b_i}(v,t) - \sum_{(i,j) \in E_{\textrm{down}}} 1\\
        &= \sum_{i = 1}^n c_i\ECT_{b_i}(v,t) - \sum_{i = 1}^n (c_i-1)\ECT_{b_i}(v,t) - \sum_{(i,j) \in E_{\textrm{down}}} 1\\
        &= \sum_{i = 1}^n \ECT_{b_i}(v,t) - \sum_{(i,j) \in E_{\textrm{down}}} 1 = \ECT^A_f(v,t).\\
    \end{split}
\end{equation*}
\end{proof}

\begin{lemma}\label{thm:lm_sl}
Let $f:\mathbb{R}_{\geq 0}\to \mathbb{R}_{\geq 0}$ be any differentiable function with increasing positive derivative satisfying $f(0) = 0$. For positive numbers $L$ and $\epsilon$ consider the set
\begin{equation*}
    S(L) = \left\{(x_1,\ldots,x_k)\in\mathbb{R}^k: 0\leq x_i\leq \epsilon, \; \sum_{i=0}^k x_i = L\right\}.
\end{equation*}
If $S(L)$ is non-empty, then 
\begin{equation*}
    \sum_{i=1}^k f(x_i)\leq Lf(\epsilon)/\epsilon
\end{equation*}
on $S(L)$ (note that $S(L)$ is always non-empty if $k > L/\epsilon$).
\end{lemma}

\begin{proof}
Let $a\leq c$ and $b \geq 0$. We have
\begin{equation*}
    f(b + c) - f(a + b) - \big(f(c)-f(a)\big) = \int_a^c f'(b + t) - f'(t)\;\mathrm{d}t \geq 0
\end{equation*}
and so
\begin{equation}
\label{eqn:increasingderivativeinequality}
    f(a + b) + f(c) \leq f(a) + f(b + c) \quad \textrm{when} \quad a\leq c \textrm{ and } b \geq 0.
\end{equation}

Suppose $S(L)$ is non-empty. Since $S(L)$ is compact, $f$ must attain a maximum on $S(L)$. Pick any such maximiser $x = (x_1,\ldots,x_k)\in S(L)$. By potentially reordering entries, we may assume the $x_i$ are in decreasing order without affecting the value of $\sum_i f(x_i)$. Let $j$ be the smallest index with $x_j \neq \epsilon$ and $l$ be the largest index with $x_l$ not equal to zero.

If $l > j$, let $m = \min(x_l, \epsilon - x_j)$. Equation (\ref{eqn:increasingderivativeinequality}) shows that if we replace $x_j$ with $x_j + m$ and replace $x_l$ with $x_l - m$, the value of $\sum_i f(x_i)$ does not decrease. Therefore, by applying this replacement procedure several times, we can always find a maximiser of $\sum_i f(x_i)$ on $S(L)$ with $j \geq l$. This condition forces the value of $\sum_i f(x_i)$ to be
\begin{equation*}
    \lfloor{L/\epsilon}\rfloor f(\epsilon) + f\big(L - \lfloor{L/\epsilon}\rfloor\epsilon\big).
\end{equation*}
If $L$ is divisible by $\epsilon$, the result is immediate. Otherwise, since $f$ is convex,
\begin{equation*}
    \begin{split}
        f(L - \lfloor{L/\epsilon}\rfloor\epsilon) &= f\Big(\big(\lceil L/\epsilon\rceil-L/\epsilon\big)0 + \big(L/\epsilon - \lfloor{L/\epsilon}\rfloor\big)\epsilon\Big)\\
        &\leq \big(\lceil L/\epsilon\rceil-L/\epsilon\big)f(0) + \big(L/\epsilon - \lfloor{L/\epsilon}\rfloor\big)f(\epsilon)\\
        &= \big(L/\epsilon - \lfloor{L/\epsilon}\rfloor\big)f(\epsilon).
    \end{split}
\end{equation*}
Thus
\begin{equation*}
    \lfloor{L/\epsilon}\rfloor f(\epsilon) + f\big(L - \lfloor{L/\epsilon}\rfloor\epsilon\big) \leq \lfloor{L/\epsilon}\rfloor f(\epsilon) + \big(L/\epsilon - \lfloor{L/\epsilon}\rfloor\big)f(\epsilon) = Lf(\epsilon)/\epsilon.
\end{equation*}
Since the value on the left is the maximum of $\sum_i f(x_i)$ on $S(L)$, we are done.
\end{proof}

\begin{proof}[Proof of Theorem \ref{thm:interpolationstable}]
For each $e = (i,j) \in E$, we let $b_e= b_{ij}$. Each $e\in E$ corresponds to some open interval in $Z-A$.
We can always fix a finer CW structure $Z^\dagger = (Z, A, \{\Phi_e\}_{e\in E})$ of $Z$, and still have that $f\in\mathcal{F}^r(Z^\dagger,d)$. Let $\gamma_e = f\circ \Phi_e$. Adopting the notation of Lemma \ref{lem:ECTapproxeqn}, induction on the number of elements in $E$ with Lemma \ref{lem:glue} applied to $Z^\dagger$ gives that for fixed $v$,
\begin{equation*}
    \ECT_f(v,t) = \sum_{e\in E} \ECT_{\gamma_e}(v,t) - \sum_{i=1}^n (c_i-1)\ECT_{b_i}(v,t),
\end{equation*}
for almost all $t$.

Therefore, again fixing $v$ and using Lemma \ref{lem:ECTapproxeqn},
\begin{equation}
\label{eqn:splitinterpolatedECT}
    \begin{split}
    &\int_\mathbb{R}\big|\ECT_f(v,t) - \ECT^A_f(v,t)\big|\; \mathrm{d}t\\
    = &\int_\mathbb{R} \bigg|\sum_{e\in E} \ECT_{\gamma_e}(v,t) - \sum_{i=1}^n (c_i-1)\ECT_{b_i}(v,t) - \Big[ \sum_{e \in E} \ECT_{b_{e}}(v,t) - \sum_{i = 1}^n (c_i-1)\ECT_{b_i}(v,t)\Big] \bigg|\; \mathrm{d}t\\
    =&\int_\mathbb{R} \Big|\sum_{e\in E} \ECT_{\gamma_e}(v,t) - \sum_{e \in E} \ECT_{b_{e}}(v,t) \Big|\; \mathrm{d}t\\
    \leq &\sum_{e\in E}\int_\mathbb{R} \Big|\ECT_{\gamma_e}(v,t) - \ECT_{b_{e}}(v,t) \Big|\; \mathrm{d}t.
    \end{split}
\end{equation}

Focusing on any particular $e = (i,j)\in E$, let $d_1$ be the minimum of $\langle \gamma_e(s),v\rangle$ over $s$, and $d_4$ be the maximum of the same function over $s$. Let $d_2 = \min(\langle b_i,v\rangle, \langle b_j,v\rangle))$ and $d_3 = \max(\langle b_i,v\rangle),\langle b_j,v\rangle))$. It follows that $d_1\leq d_2 \leq d_3 \leq d_4$.

Since subsets of $I$ and $b_e$ always consist of contractible components, $\ECT_{\gamma_e}$ and $\ECT_{b_e}$ never have negative values. We have
\begin{equation*}
    \begin{split}
        t\geq d_1 \quad &\implies \quad \ECT_{\gamma_e}(v,t)\geq1,\\
        t\geq d_4 \quad &\implies \quad \ECT_{\gamma_e}(v,t) = 1,\\
        t < d_1 \quad &\implies \quad \ECT_{\gamma_e}(v,t) = 0,\\
        t\geq d_2 \quad &\implies \quad \ECT_{b_e}(v,t) = 1,\\
        t < d_2 \quad &\implies \quad \ECT_{b_e}(v,t) = 0.
    \end{split}
\end{equation*}
Combining these observations, we see
\begin{equation*}
    \begin{split}
        \int_\mathbb{R} \Big|\ECT_{\gamma_e}(v,t) - \ECT_{b_{e}}(v,t) \Big|\; \mathrm{d}t &= \int_{d_1}^{d_4} \ECT_{\gamma_e}(v,t) -\ECT_{b_e}(v,t)\;\mathrm{d}t\\
        & \leq \int_{d_1}^{d_4} \ECT_{\gamma_e}(v,t)\;\mathrm{d}t - (d_3 - d_2).\\
    \end{split}
\end{equation*}
After applying a rotation, we may assume that $v = (0,1,0,\ldots,0)$. After applying another rotation about $v$ we may assume that $b_e$ is parallel to the plane spanned by the first two coordinates. Let $l_e$ be the arc length of $\gamma_e$. By Proposition \ref{prop:curvbound}, the length of $b_e$ is at least $l_e - M^2l_e^3/24$.
Suppose that the line segment $b_e$ meets the hyperplane perpendicular to $v$ at an angle $\theta \in [0,\pi/2]$.

Applying Proposition \ref{prop:ectbound} and Lemma \ref{lem:varbound} to this scenario, we observe
\begin{equation*}
    \int_{d_1}^{d_4} \ECT_{\gamma_e}(v,t)\;\mathrm{d}t - (d_3 - d_2)\leq \sqrt{l_e^2 - \left(l_e-\frac{M^2}{24}l_e^3\right)^2\cos^2\theta} - \left(l_e-\frac{M^2}{24}l_e^3\right)\sin\theta.
\end{equation*}
We refer to the right side of this inequality as $f(\theta)$. Let $G = l_e - M^2l_e^3/24$. $G$ is positive since $l_e < \epsilon < \pi/M < \sqrt{24}/M$. We have
\begin{equation*}
    f'(\theta) = \frac{G^2\sin\theta\cos\theta}{\sqrt{l_e^2-G^2\cos^2\theta}} -G\cos\theta    .
\end{equation*}
A routine calculation shows that $f'$ is either zero only when $\theta=\pi/2$ or for every $\theta$. Meanwhile $f'(0) = -G$. Since this value is negative, $f$ must be maximised at $\theta = 0$. Hence,
\begin{equation}
\label{eqn:interpolatedlineguess}
    \begin{split}
        \int_\mathbb{R} \left|\ECT_{\gamma_e}(v,t) - \ECT_{b_{e}}(v,t) \right|\; \mathrm{d}t&\leq\int_{d_1}^{d_4} \ECT_{\gamma_e}(v,t)\;\mathrm{d}t - (d_3 - d_2)\\
        &\leq \sqrt{l_e^2 - \left(l_e-\frac{M^2}{24}l_e^3\right)^2}\\
        & = \sqrt{\frac{M^2}{12}l_e^4-\frac{M^4}{24^2}l_e^6}\\
        &\leq\frac{M}{\sqrt{12}}l_e^{2}.
    \end{split}
\end{equation}
For $\lambda\in\Lambda$, let $L_\lambda$ denote the arc length of $f\circ\Phi_\lambda$. Now for $\lambda\in\Lambda$, denote by $\Gamma(\lambda)$ the submultiset of $e\in E$ such that $\im\Phi_e$ is a subset of $\im\Phi_\lambda$. By Equations (\ref{eqn:splitinterpolatedECT}) and (\ref{eqn:interpolatedlineguess}), along with Lemma \ref{thm:lm_sl}, we get that
\begin{equation*}
    \begin{split}
    \int_\mathbb{R}\Big|\ECT_f(v,t) - \ECT^A_f(v,t)\Big|\; \mathrm{d}t &\leq\frac{M}{\sqrt{12}}\sum_{e\in E}l_e^2\\
    & = \frac{M}{\sqrt{12}}\sum_{\lambda\in \Lambda}\sum_{e\in\Gamma(\lambda)}l_e^2\\
    & \leq \frac{M}{\sqrt{12}}\sum_{\lambda\in\Lambda}L_\lambda\epsilon^2/\epsilon\\
    &= \frac{ML\epsilon}{\sqrt{12}}.
    \end{split}
\end{equation*}
Since this bound holds for any $v$, we are done.
\end{proof}

\section{ECT Stability of Random Data}\label{sec:ect_prob_stab}

In this section, we consider observations taken from an embedded finite one-dimensional CW complex $Z$ which are perturbed by ambient Gaussian noise. We show that the Gaussian smoothing of these observations converges to satisfy the assumptions of Proposition \ref{prop:funcstab}. 
In particular, we show that the ECT and SECT of the Gaussian smoothing give consistent estimators of the ECT and SECT of $Z$, respectively. To provide the theorems, we first need to introduce technical conditions on the kernel we use in the Gaussian smoothing:

\begin{definition}[Definition 5 in \cite{koepernik2021consistency}]
Let $Z$ be a topological space and $k:Z\times Z\to\mathbb{R}$ be a continuous kernel. Define
$$d_k(t,s)=\sqrt{k(t,t)+k(s,s)-2k(t,s)}.$$
For any $\varepsilon>0$ let $N(Z, \varepsilon, d_k)$ be the minimal numbers of $d_k$-balls with radius $\varepsilon$ needed to cover $Z$. Then define
$$J(Z,d_k)=\int_0^\infty \sqrt{\log N(Z, \varepsilon, d_k)} \,\mathrm{d}\varepsilon.$$
\end{definition}

\begin{definition}
    Let $Z^* = (Z, Z_0,\{\Phi_\lambda\}_{\lambda\in\Lambda})$ be a connected finite one-dimensional CW complex with some fixed cellular decomposition. Let $k:Z\times Z\to\RR$ be a continuous kernel. We say $k$ is $r$-times differentiable on $Z^*$ if 
    \begin{enumerate}
        \item for each $\lambda\in\Lambda$ the map $k^\lambda:I\times I\to\RR$ given by $(s,t)\mapsto k(\Phi_\lambda(s),\Phi_\lambda(t))$ is $r$-times continuously differentiable and
        \item for each $\lambda\in\Lambda$ and $z\in Z$ the map $k^{\lambda,z}:I\to\RR$ given by $s\mapsto k(\Phi_\lambda(s),z)$ is $r$-times continuously differentiable.
    \end{enumerate}
    Differentiability is defined by one-sided limits at the boundaries of $I\times I$ and $I$.
\end{definition}

\begin{remark}
    For a given connected finite 1
    one-dimensional CW complex $Z^* = (Z, Z_0,\{\Phi_\lambda\}_{\lambda\in\Lambda})$ with fixed cellular composition there is a straightforward way to construct an $r$-times differentiable kernel on $Z^*$: let $f:Z\to\RR^d$ be a continuous function such that $f\circ\Phi_\lambda$ is $r$-times differentiable for each $\lambda\in\Lambda$. Then if $k$ is an $r$-times differentiable kernel on $\RR^d$, it follows that $k'(s,t):=k(f(s),f(t))$ is an $r$-times differentiable kernel on $Z$ by the chain rule.

    While it might be tempting to define a geodesic distance on $Z$ and then apply a stationary kernel (such as the Gaussian kernel) to this distance, it should be noted that, even in the case of $Z$ being a manifold, the resulting function does not give a positive-definite kernel in general \cite{feragen2015geodesic}.
\end{remark}

We can now state the first theorem of this section:

\begin{theorem}\label{thm:smoothing}
Let $Z^* = (Z, Z_0,\{\Phi_\lambda\}_{\lambda\in\Lambda})$ be a connected finite one-dimensional CW complex with some fixed cellular structure.
Let $k:Z\times Z\to\mathbb{R}$ be a continuous, four-times differentiable kernel on $Z^*$. Assume $k$ satisfies $J(Z,d_k)<\infty$.

Let $f:Z\to\mathbb{R}$ be a function in the RKHS of $k$. Let $\mathbf{a}\subset Z$ be a sequence which is dense. Denote by $\mathbf{a}_n$ the first $n$ terms of $\mathbf{a}$ and by $a_n$ the $n$-th term of $\mathbf{a}$. Let $\hat{f}_n$ denote the Gaussian smoothing of $f$ based on observations $y_i=f(a_i)+\zeta_i$ using kernel $k$, where $i=1,...,n$ and $\zeta_i\sim\mathcal{N}(0,\sigma)$ i.i.d. for some $\sigma>0$. Then
$$\mathbb{E}\left[\left\|\hat{f}_n(t,\mathbf{a}_n, f)-f(t)\right\|_\infty\right]\to 0$$
as $n\to\infty$.
Moreover, for each $\lambda\in\Lambda$ define $\hat{f}_{n,\lambda}(t)=\hat{f}_n(\Phi_\lambda(t))$ and  $f_{\lambda}(t)=f(\Phi_\lambda(t))$. Then
$$\mathbb{E}\left[\left|V\left(\hat{f}_{n,\lambda}\right) -V\left(f_\lambda\right)\right|^2\right]\to0$$ 
on each 1-cell of $Z^*$ as $n\to\infty$; i.e. the variation of $\hat{f}_{n,\lambda}$ converges to the variation of ${f}_{\lambda}$ in mean square.
\end{theorem}

When proving the above result, we write $k_x$ and $k_y$ for the partial derivatives in the first and second components, respectively, and $K_x$ and $K_y$ for their corresponding Gram matrices. In particular, for fixed $t\in I$, $\lambda\in\Lambda$, and $\mathbf{a}$, we write $K^\lambda_x(t,\mathbf{a}_n)=[k_x^{\lambda,a_1}(t),...,k_x^{\lambda,a_n}(t)]$ and $K^\lambda_y$ for its transpose. A repeated subscript indicates repeated differentiation in that variable.

Let $g:X\to\RR$ be a GP with kernel $k$ and a deterministic function $h:X'\to X$. Then $g\circ h$ is a GP with kernel $k'(x,y):=k(h(x),h(y))$ for all $x,y\in X'$. This insight immediately follows from the definition of a GP in Definition \ref{def:gp_def}. In particular, for the GP $f$ in the statement of this theorem and any $\lambda\in\Lambda$, the composition $f\circ\Phi_\lambda$ is a GP for any number of observations $n$.

The derivative of a Gaussian process on $I$ with a differentiable kernel is almost surely differentiable. As differentiation is a linear operator, the derivative of a Gaussian process is again a Gaussian process in such a case \cite{rasmussen2003gaussian}. In particular, this derivative GP has kernel $k_{xy}$ and for any $t\in I$ we have the joint distribution
\begin{equation}
\begin{bmatrix}
    g(t)\\ g'(t)
\end{bmatrix}\sim\mathcal{N}\left(
\begin{bmatrix}
    \mu(t)\\ \mu'(t)
\end{bmatrix},
\begin{bmatrix}
    k(t,t) & k_y(t,t)\\
    k_x(t,t) & k_{xy}(t,t)
\end{bmatrix}\right).\label{eq:joint_dist}
\end{equation}

In Theorem \ref{thm:smoothing}, we consider the GP regression of $f_\lambda$ based on observations at $\mathbf{a}$ for fixed $\lambda$. Even if not all elements in the sequence $\mathbf{a}$ need to be in the image of $\Phi_\lambda$, the GP posterior pre-composed with $\Phi_\lambda$ defines a GP on $I$. We are interested in the convergence of the derivative of this GP.

For fixed $\lambda\in\Lambda$, we denote the variance of this derivative GP at $t\in I$ by $v'_{n,\lambda}(t)$ (which is not the same as the derivative of $v_{n,\lambda}(t)$ in $t$). In particular, we have
$$v'_{n,\lambda}(t)=k^\lambda_{xy}(t,t)-K^\lambda_{x}(t,\mathbf{a}_n)(K(\mathbf{a}_n,\mathbf{a}_n)+\sigma^2I)^{-1}K^\lambda_{y}(\mathbf{a}_n,t).$$

\begin{lemma}\label{lem:mono}
    Given the Gaussian processes of Theorem \ref{thm:smoothing}, we get that
    for each $\lambda\in\Lambda$ the $v'_{n,\lambda}$ satisfy
    $$v'_{n,\lambda}(t)=\mathbb{E}\left[\left|\hat{f}'_{n,\lambda}(t,\mathbf{a}_n, f)-f'_\lambda(t)\right|^2\right].$$
    Furthermore, $v'_{n,\lambda}(t)$ is monotonically decreasing in $n$ for all $t\in I$.
\end{lemma}
\begin{proof}
    The first statement follows from Lemma 11 of \cite{koepernik2021consistency}. In particular, 
    \begin{align*}
        \mathbb{E}_{f'_\lambda}\left[\left|\hat{f}'_{n,\lambda}(t)-f'_\lambda(t)\right|^2\right] &= \mathbb{E}_{\boldsymbol\zeta_n}\left[\mathbb{E}_{f'_\lambda}\left[\left|\hat{f}'_{n,\lambda}(t)-f'_\lambda(t)\right|^2\right]\Big\vert \mathbf{f}(\mathbf{a}_n)+\boldsymbol\zeta_n = \mathbf{y}_n\right]\\
        &=\mathbb{E}_{\boldsymbol\zeta_n}\left[\mathbb{E}_{f'_\lambda}\left[\left|f'_\lambda(t)-\mathbb{E}_{f'_\lambda}[f'_\lambda(t)\vert \mathbf{f}(\mathbf{a}_n)+\boldsymbol\zeta_n]\right|^2\right]\Big\vert \mathbf{f}(\mathbf{a}_n)+\boldsymbol\zeta_n = \mathbf{y}_n\right]\\
        &=\mathbb{E}_{\boldsymbol\zeta_n}[\mathrm{Var}(f'_\lambda(t)\vert \mathbf{f}(\mathbf{a}_n)+\boldsymbol\zeta_n)]\\
        &=\mathbb{E}_{\boldsymbol\zeta_n}[v'_{n,\lambda}(t)]= v'_{n,\lambda}(t).
    \end{align*}
    
    For the second statement, we can write 
    $$v'_{n,\lambda}(t)=k^\lambda_{xy}(t,t)-K^\lambda_x(t,\mathbf{a}_{n+1})
        \begin{bmatrix}
            B_n^{-1} & \mathbf{0}_{n\times 1} \\
            \mathbf{0}_{1\times n} & 0
        \end{bmatrix}
    K^\lambda_y(\mathbf{a}_{n+1},t),$$
    where $B_n:=(K(\mathbf{a}_{n},\mathbf{a}_{n})+\sigma^2I_n)$.
    Using the bordering method to obtain an expression for $B_{n+1}^{-1}$  in terms of $B_n$, we get
    $$v'_{n,\lambda}(t)-v'_{n+1,\lambda}(t)$$
    $$=\nu^{-1}K^\lambda_x(t,\mathbf{a}_{n+1})
        \begin{bmatrix}
            B_n^{-1}K(\mathbf{a}_n,a_{n+1})K(a_{n+1},\mathbf{a}_n)B_n^{-1} & -B_n^{-1}K(\mathbf{a}_n,a_{n+1}) \\
            -K(a_{n+1},\mathbf{a}_n)B_n^{-1} & 1
        \end{bmatrix}
    K^\lambda_y(\mathbf{a}_{n+1},t)$$
    $$=\nu^{-1}((K^\lambda_x(t,\mathbf{a}_n)B_n^{-1}K(\mathbf{a}_n,a_{n+1}))^2-2K^\lambda_x(t,a_{n+1})K^\lambda_x(t,\mathbf{a}_n)B_n^{-1}K(\mathbf{a}_n,a_{n+1})+k^\lambda_x(t,a_{n+1})^2)$$
    $$=\nu^{-1}(K^\lambda_x(t,\mathbf{a}_n)B_n^{-1}K(\mathbf{a}_n,a_{n+1})-k^\lambda_x(t,a_{n+1}))^2,$$
    where 
    $\nu:=k(a_{n+1},a_{n+1})+\sigma-K(a_{n+1},\mathbf{a}_n)B_n^{-1}K(\mathbf{a}_n,a_{n+1})=v_n(a_{n+1})+\sigma$
    is the Schur complement of $B_n$ inside $B_{n+1}$. As the $\nu$ is the Schur complement of a positive-definite matrix inside a positive-definite matrix, it is positive. As the second factor in the final line above is a square and thus positive too, we conclude that the sequence of functions $v'_{n,\lambda}(t)$ is monotonically decreasing.
\end{proof}

\begin{proof}[Proof of Theorem \ref{thm:smoothing}]

The first statement follows from Theorem 8 in \cite{koepernik2021consistency}.

To prove the remainder of the theorem, we recall from Equation (\ref{eq:joint_dist}) that the covariance matrix of the distribution of $\left({f}_{\lambda}(t), f'_{\lambda}(t)\right)^T$ given $n$ noisy observations of ${f}$ is
$$\begin{bmatrix}
k(t,t)-K(t,\mathbf{a}_n)B_n^{-1}K(\mathbf{a}_n,t) & k^\lambda_y(t,t)-K(t,\mathbf{a}_n)B_n^{-1}K^\lambda_y(\mathbf{a}_n,t) \\ k^\lambda_x(t,t)-K^\lambda_x(t,\mathbf{a}_n)B_n^{-1}K(\mathbf{a}_n,t) & k^\lambda_{xy}(t,t)- K^\lambda_x(t,\mathbf{a}_n)B_n^{-1}K^\lambda_y(\mathbf{a}_n,t)
\end{bmatrix}.$$
As this matrix needs to be positive-definite, by taking the determinant and using the symmetry of $k$ we get
\begin{equation}
(k(t,t)-K(t,\mathbf{a}_n)B_n^{-1}K(\mathbf{a}_n,t))(k^\lambda_{xy}(t,t)- K^\lambda_x(t,\mathbf{a}_n)B_n^{-1}K^\lambda_y(\mathbf{a}_n,t)) \label{eq:unif_conv_bound}
\end{equation}
\begin{equation}
\geq (k^\lambda_x(t,t)-K^\lambda_x(t,\mathbf{a}_n)B_n^{-1}K(\mathbf{a}_n,t))^2\geq 0.\label{eq:unif_conv}
\end{equation}
Thus, $k^\lambda_x(t,t)-K^\lambda_x(t,\mathbf{a}_n)B_n^{-1}K(\mathbf{a}_n,t)\to0$ uniformly on $I$ as the first factor of (\ref{eq:unif_conv_bound}) converges uniformly by Proposition 10 of \cite{koepernik2021consistency} and the second factor of (\ref{eq:unif_conv_bound}) is bounded by the monotonicity established in Lemma \ref{lem:mono} and the compactness of $I$. Repeating the same procedure with $\hat{f}''_{n,\lambda}$ in place of  $\hat{f}'_{n,\lambda}$ gives $k^\lambda_{xx}(t,t)-K^\lambda_{xx}(t,\mathbf{a}_n)B_n^{-1}K(\mathbf{a}_n,t)\to0$ uniformly: in this case, the second factor is $k^\lambda_{xxyy}(t,t)- K^\lambda_{xx}(t,\mathbf{a}_n)B_n^{-1}K^\lambda_{yy}(\mathbf{a}_n,t)$, which equals $v_{n,\lambda}''(t)$, the variance of the second derivative of the GP $f_\lambda$. We can show that $v_{n,\lambda}''(t)$ monotonically decreases by a proof analogous to the case $v_{n,\lambda}'(t)$ given in Lemma \ref{lem:mono}. For $v_{n,\lambda}''(t)$ to be well-defined we require $k$ to be four times differentiable.

Then, by Jensen's inequality and Lemma \ref{lem:mono}, we can bound the expected value of the squared difference $V\left(\hat{f}_{n,\lambda}\right)-V(f_\lambda)$:
$$\mathbb{E}\left[\left|\int_0^1\left|f'_\lambda(t)\right|-\left|\hat{f}'_{n,\lambda}(t, \mathbf{a}_n, f)\right|\,\mathrm{d}t\right|^2\right]\leq \mathbb{E}\left[\left(\int_0^1\left|\left|f'_\lambda(t)\right|-\left|\hat{f}'_{n,\lambda}(t, \mathbf{a}_n, f)\right|\right|\,\mathrm{d}t\right)^2\right]$$
$$\leq\mathbb{E}\left[\int_0^1\left|\left|f'_\lambda(t)\right|-\left|\hat{f}'_{n,\lambda}(t, \mathbf{a}_n, f)\right|\right|^2\,\mathrm{d}t\right]=\int_0^1\mathbb{E}\left[\left|\left|f'_\lambda(t)\right|-\left|\hat{f}'_{n,\lambda}(t, \mathbf{a}_n, f)\right|\right|^2\right]\,\mathrm{d}t=\int_0^1{v'_{n,\lambda}(t)}\,\mathrm{d}t$$
$$=\left[k^\lambda_x(t,t)-K^\lambda_x(t,\mathbf{a}_n)B_n^{-1}K(\mathbf{a}_n,t)-\int_0^tk^\lambda_{xx}(s,s)-K^\lambda_{xx}(s,\mathbf{a}_n)B_n^{-1}K(\mathbf{a}_n,s)\,\mathrm{d}s\right]_0^1.$$
The final equation above converges to $0$ as $n\to\infty$, as both the left-hand term and the function under in the integral of the right-hand term in the above difference converge uniformly to 0 by Equation (\ref{eq:unif_conv}) and its analogue for $v_{n,\lambda}''(t)$.
\end{proof}

It follows that the ECT of the interpolation of the Gaussian smoothing $\hat{f}_n$ of $f$, denoted $\ECT^{\mathbf{a}_m}_{\hat{f}_n}$, is a consistent estimator of the ECT of $X$:

\begin{theorem}\label{thm:combined}
    Let $Z^* = (Z, Z_0,\{\Phi_\lambda\}_{\lambda\in\Lambda})$ be a finite one-dimensional CW complex with some fixed cellular structure and $f:Z \to X \subseteq \mathbb{R}^d$ be a $C^2$ homeomorphism with bounded curvature. Further, assume that all components of $f$ are functions in the RKHS of $k$, where $k$ is a kernel satisfying the assumptions of Theorem \ref{thm:smoothing}. Moreover, assume that $\left\|f'_\lambda(t)\right\|_2=L_\lambda$ is constant on all 1-cells $\lambda\in\Lambda$. Let $\mathbf{a}$ be a sequence in $Z$ which is compatible with $Z^*$ and dense for $f$. Let
    $$f(t):=\left(f^1(t),...,f^d(t)\right)^T, \qquad\hat{f}_n:=\left(\hat{f}^1_n,...,\hat{f}^d_n\right)^T,$$
    where for $j=1,...,d$ and $i=1,...,n$ the function $\hat{f}^j_n$ is the Gaussian smoothing of $f^j$ given observations $y_{ij}=f^j(a_i)+\zeta_{ij}$ using kernel $k$ and $\zeta_{ij}\sim\mathcal{N}(0,\sigma_j)$ i.i.d for some $\sigma_j>0$. Then for each $\varepsilon>0$
    $$\lim_{n\to\infty}\mathbb{P}\left(\left\|\ECT_{\hat{f}_n}- \ECT_{\,f}\right\|<\varepsilon\right)\to 1.$$
\end{theorem}

Note that as $f$ is a homeomorphism, $\ECT_{f}=\ECT_{\,\im f}=\ECT_X$ and we thus have constructed a consistent estimator for $\ECT_X$. If for given observations $\mathbf{y}_n$ the curvature of $\hat{f}_n$ is bounded on each 1-cell, we can approximate $\ECT_{\hat{f}_n}$ by $\ECT_{\hat{f}_n}^{\mathbf{a}_m}$ arbitrarily closely for a sufficiently large $m$ by Theorem \ref{thm:interpolationstable}. We conjecture that for sufficiently well-behaved kernels $k$, $\ECT_{\hat{f}_n}^{\mathbf{a}_m}$ converges to $\ECT_{\hat{f}_n}$ in probability, where $m$ is some function in $n$. Proving this conjecture will involve bounding the curvature with high probability and is beyond the scope of this paper.

\begin{lemma}\label{lem:arc_length}
    Let $f$ and $\hat{f}_n$ be as in the statement of Theorem \ref{thm:combined}. Denote the arc-lengths of $\hat{f}_{\lambda,n}:=\hat{f}_n\circ\Phi_\lambda$ and $f_\lambda:=f\circ\Phi_\lambda$ by $L_{n,\lambda}$ and $L_\lambda$ respectively for each $\lambda\in\Lambda$. Then $L_{n,\lambda}\to L_\lambda$ and
    $$\int_0^1\left|\left\|f'_{n,\lambda}(t)\right\|_2-\left\|f'_{\lambda}(t)\right\|_2\right|\,\mathrm{d}t\to 0$$
    in probability.
\end{lemma}
\begin{proof}
    First, note that $\sqrt{x+y}\leq\sqrt{x}+\sqrt{y}$ and $|\sqrt{x}-\sqrt{y}| \leq \sqrt{|x-y|}$ for all $x,y\geq 0$. We have
    \begin{align*}
        \left|\int_0^1\left\|f'_{n,\lambda}(t)\right\|_2-\left\|f'_{\lambda}(t)\right\|_2\mathrm{d}t\right| &\leq \int_0^1\left|\left\|f'_{n,\lambda}(t)\right\|_2-\left\|f'_{\lambda}(t)\right\|_2\right|\,\mathrm{d}t\\
        &\leq \int_0^1\sqrt{\left|\big\|f'_{n,\lambda}(t)\big\|^2_2-\left\|f'_{\lambda}(t)\right\|^2_2\right|}\,\mathrm{d}t\\
        &\leq\sum_{j=1}^d\int_0^1\sqrt{\left||(f^j_{n,\lambda})'(t)|^2-|(f^j_{\lambda})'(t)|^2\right|}\,\mathrm{d}t\\
        &\leq\sum_{j=1}^d\left[\int_0^1|(f^j_{n,\lambda})'(t)|+|(f^j_{\lambda})'(t)|\mathrm{d}t\right]\!\!\left[\int_0^1\left||(f^j_{n,\lambda})'(t)\right|-\left|(f^j_{\lambda})'(t)|\right|\mathrm{d}t\right].
    \end{align*}
    The first inequality follows from the triangle inequality for integrals. The second and third inequalities follow from the inequalities for square roots introduced at the start of the proof. The final inequality is the Cauchy-Schwarz inequality for integrals.

    The first factor in the final line converges to $2V(f_\lambda)$ and the second factor converges to 0 for each $j=1,...,d$ by Theorem \ref{thm:smoothing} in probability.
\end{proof}

\begin{proof}[Proof of Theorem \ref{thm:combined}]
    We recall that convergence in mean implies convergence in probability.
    By applying Theorem \ref{thm:smoothing} to each component of $f$, we find that $\hat{f}_n$ converges to $f$ in mean in the $\infty$-norm.
    Note that $\hat{f}_{n,\lambda}$ need not be parameterised to constant velocity. Denote the arc length of $\hat{f}_{n,\lambda}$ by $L_{n,\lambda}$ and the arc length of  $f_{\lambda}$ by $L_{\lambda}$. By Lemma \ref{lem:arc_length}, $L_{n,\lambda}\to L_\lambda$ in probability for each $\lambda\in\Lambda$. Let $s_{n,\lambda}$ be the re-parametrisation of $\hat{f}_{n,\lambda}$ to constant-velocity on $I$, which is given by
    $$s_{n,\lambda}(t)=\frac{1}{L_{n,\lambda}}\int_0^t\left\|\hat{f}'_{n,\lambda}(x)\right\|_2\,\mathrm{d}x$$
    and satisfies $\left\|\left(\hat{f}_{n,\lambda}\circ s^{-1}\right)'(x)\right\|_2=1$ on $I$. Thus,
    \begin{equation*}
        |s_{n,\lambda}(t)-t|=\left|\int_0^t\left\|\hat{f}'_{n,\lambda}(x)\right\|_2/L_{n,\lambda}-1\,\mathrm{d}x\right|\leq\left|\int_0^t\left|\left\|\hat{f}'_{n,\lambda}(x)\right\|_2/L_{n,\lambda}-\left\|f'_\lambda(x)\right\|_2/L_{\lambda}\right|\mathrm{d}x\right|
    \end{equation*}
    \begin{equation}
        \leq\frac{1}{L_{n,\lambda}}\int_0^1\left|\left\|f'_{n,\lambda}(x)\right\|_2-\left\|f'_\lambda(x)\right\|_2\right|\,\mathrm{d}x+\left\|f'_\lambda(x)\right\|_2\left|\frac{1}{L_{\lambda,n}}-\frac{1}{L_\lambda}\right|.
        \label{eq:rp0}
    \end{equation}
    As both terms in Equation (\ref{eq:rp0}) converge to 0 in probability independently of $t$ by Lemma \ref{lem:arc_length}, we get $\|s_{n,\lambda}(t)-t\|_\infty\xrightarrow[]{p.}0$.

    We then define $s_n:Z\to Z$ as
    \begin{equation*}
        s_n(z) =
        \begin{cases}
            z & \text{if $z\in Z_0$},\\
            \left(\Phi_\lambda\circ s_{n,\lambda}^{-1}\circ\Phi_\lambda^{-1}\right)(z) & \text{if $z\in\Phi_\lambda((0,1))$.}
        \end{cases}
    \end{equation*}
    The map $s_n$ is continuous as each $s_{n,\lambda}^{-1}$ is continuous, $s_{n,\lambda}^{-1}(0)=0$ and $s_{n,\lambda}^{-1}(1)=1$.
    
    The result of the theorem then follows from Proposition \ref{prop:funcstab}: 
    \begin{equation}
        \left\|\ECT_{\hat{f}_{n}} - \ECT_{\,f}\right\|\leq \left\|\ECT_{\,\hat{f}_n}-\ECT_{\,\hat{f}_n\circ s_n}\right\| + \left\|\ECT_{\,\hat{f}_n\circ s_n}-\ECT_{f}\right\|. \label{eq:final_conv}
    \end{equation}
    Note that the first term is 0 as re-parametrisation does not change the image of a function. For the second term, we find that $\hat{f}_n\circ s_n$ converges to satisfy the conditions such that Proposition \ref{prop:funcstab} yields increasingly tight bounds: the arc lengths of $\hat{f}_{n}\circ s_n\circ \Phi_\lambda = \hat{f}_{n,\lambda}\circ s_{n,\lambda}^{-1}$  converge to those of $f\circ\Phi_\lambda = f_\lambda$ by Lemma \ref{lem:arc_length} (the composition of $f_\lambda$ with $s_{n,\lambda}^{-1}$ does not change its arc length). Further, both aforementioned functions have constant velocity and 
    $$\left\|\hat{f}_n\circ s_n-f\right\|_\infty\leq\left\|\hat{f}_n\circ s_n-\hat{f}_n\right\|_\infty+\left\|\hat{f}_n-f\right\|_\infty\xrightarrow[]{p.}0.$$
    In the above, the second term converges in probability by Theorem \ref{thm:smoothing}. The first term converges in probability as
    \begin{equation*}
        \Big\| \hat{f}_{n,\lambda}\circ s_{n,\lambda}^{-1} - \hat{f}_{n,\lambda}\Big\|_\infty \leq \Big\| \hat{f}_{n,\lambda}\circ s_{n,\lambda}^{-1} - f_\lambda\circ s_{n,\lambda}^{-1} \Big\|_\infty + \Big\| f_\lambda\circ s_{n,\lambda}^{-1} - f_\lambda\Big\|_\infty  + \Big\|f_\lambda-  \hat{f}_{n,\lambda}\Big\|_\infty\xrightarrow[]{p.}0
    \end{equation*}
    on each 1-cell $\lambda\in\Lambda$. The first term converges in probability by Theorem \ref{thm:smoothing} (as re-parametrisation does not change the $\infty$-norm). Note that $f_\lambda$ is continuous on $I$, which is compact, and therefore uniformly continuous. The second term equals $\Big\| f_\lambda\circ s_{n,\lambda} - f_\lambda\Big\|_\infty$ by pre-composition with $s_{n,\lambda}(t)$ and thus converges by Equation (\ref{eq:rp0}) and the uniform continuity of $f_\lambda$. The last term converges in probability by Theorem \ref{thm:smoothing}.
\end{proof}

Furthermore, our consistency result extends to the SECT of $X$.

\begin{lemma}
Define the $\ECT$ on some interval $[-a, a]$. Assume the distance between the $\ECT$s of two shapes $X$ and $Y$ is $\delta$. Then the distance between their $\SECT$s is at most $(2a+1)\delta$.
\end{lemma}

\begin{proof}
Fix $v\in S^{d-1}$. Then
$$\|\SECT_X(v,\,\cdot\,)-\SECT_Y(v,\,\cdot\,)\|_1$$
$$=\int_{-a}^a \left|\int_{-a}^t\ECT_X(v,x)-\ECT_Y(v,x)\,\mathrm dx -\frac{t+a}{2a}\int_{-a}^a\ECT_X(v,x)-\ECT_Y(v,x)\,\mathrm dx\right|\,\mathrm{d}t$$
$$\leq\int_{-a}^a \int_{-a}^t\left|\ECT_X(v,x)-\ECT_Y(v,x)\right|\,\mathrm dx +\frac{t+a}{2a}\int_{-a}^a\left|\ECT_X(v,x)-\ECT_Y(v,x)\right|\,\mathrm dx\,\mathrm{d}t$$
$$\leq 2a\delta + \delta = (2a+1)\delta. $$
Since the above is independent of $v$, we are done.
\end{proof}

The main limitation of our results is that the topology of our embedded space is assumed to be known and the results only work for a restricted class of CW complexes. Extending our statistical estimator and the related results to perturbations in the topology of the underlying shape remains future work.

\section{Examples}\label{sec:example}

We now illustrate our methods by means of a simulated example. In our simulation, we focus on a single simple closed curve in $\RR^2$ and sample different numbers of noisy points from the curve. Our curve has been constructed by judiciously choosing complex Fourier coefficients. The samples are then taken by evenly spaced evaluations of our curve and are corrupted by adding independent multivariate Gaussian noise with mean 0 and covariance $(0.002)^2I_2$. The curve, together with the noisy samples, is visualised in Figure \ref{fig:example_curves}.

As a kernel in our Gaussian smoothing, we pick the \emph{sine-squared exponential kernel}. Assuming our curve is parameterised by $\gamma:[0,2\pi]\to\RR^2$ with $\gamma(0)=\gamma(2\pi)$, it is given by
$$k(s,t)=\exp\left(-2\sin\left(\frac{s-t}{2}\right)^2\right).$$
It satisfies the conditions of Theorems \ref{thm:smoothing} and \ref{thm:combined} (see Lemma \ref{thm:sse_bound}; it is infinitely differentiable as it is the composition of infinitely differentiable functions). Its RKHS contains the curve we generated (see Lemma \ref{thm:sse_rkhs}).

\begin{figure}
    \centering
    \includegraphics[width=\textwidth]{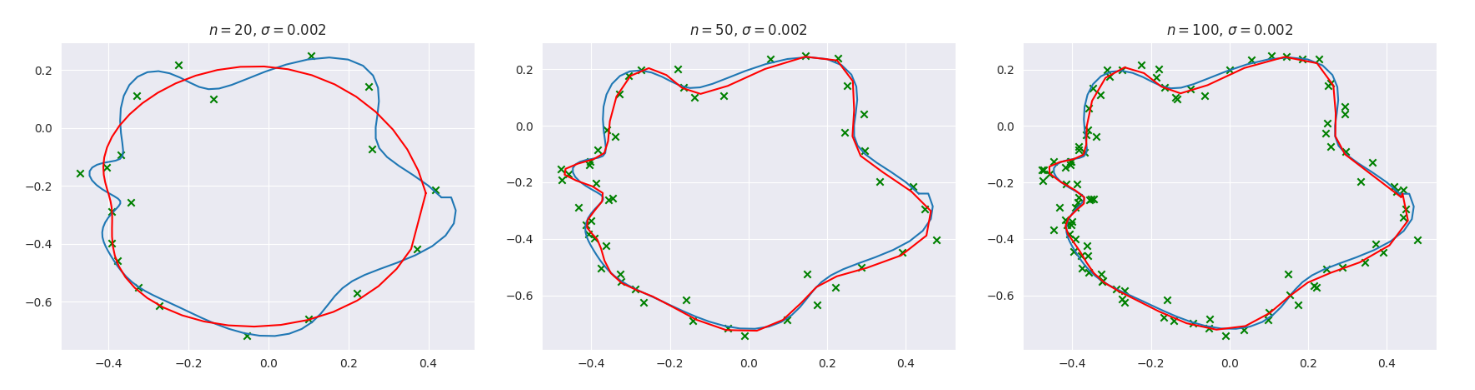}
    \caption{The Gaussian smoothings (red lines) of a simple closed curve (blue line) based on noisy samples (green crosses). The number of points is $20$ on the left panel, 50 in the middle panel and 100 in the right panel. All points have been independently corrupted with mean zero Gaussian noise with standard deviation $\sigma=0.002$ in each component.}
    \label{fig:example_curves}
\end{figure}

In Figure \ref{fig:ect_dist}, we visualise the SECT of our true curve (in a fixed direction) and compare it to the SECT of curves sampled from Gaussian process regression (GPR) posterior distributions based on 20, 50 and 100 noisy evaluations of our original curves, respectively. In addition, we plot the distributions of the distance (given by the norm introduced in Equation (\ref{eq:norm})) of the SECTs of the posterior samples with the SECT of the true curve.

In both types of plots, we see the posterior curves' mass moving closer to the true SECT, thereby illustrating the results of our theorems. We furthermore report that the distance between the SECT of the true curve and the SECTs of Gaussian smoothings are approximately 0.0627 ($n=20$), 0.0366 ($n=50$) and 0.0214 ($n=100$), respectively. However, while Figure \ref{fig:ect_dist} illustrates that our results provide a consistent estimator of the SECT, the estimator need not be unbiased.

\begin{figure}
    \centering
    \includegraphics[width=\textwidth]{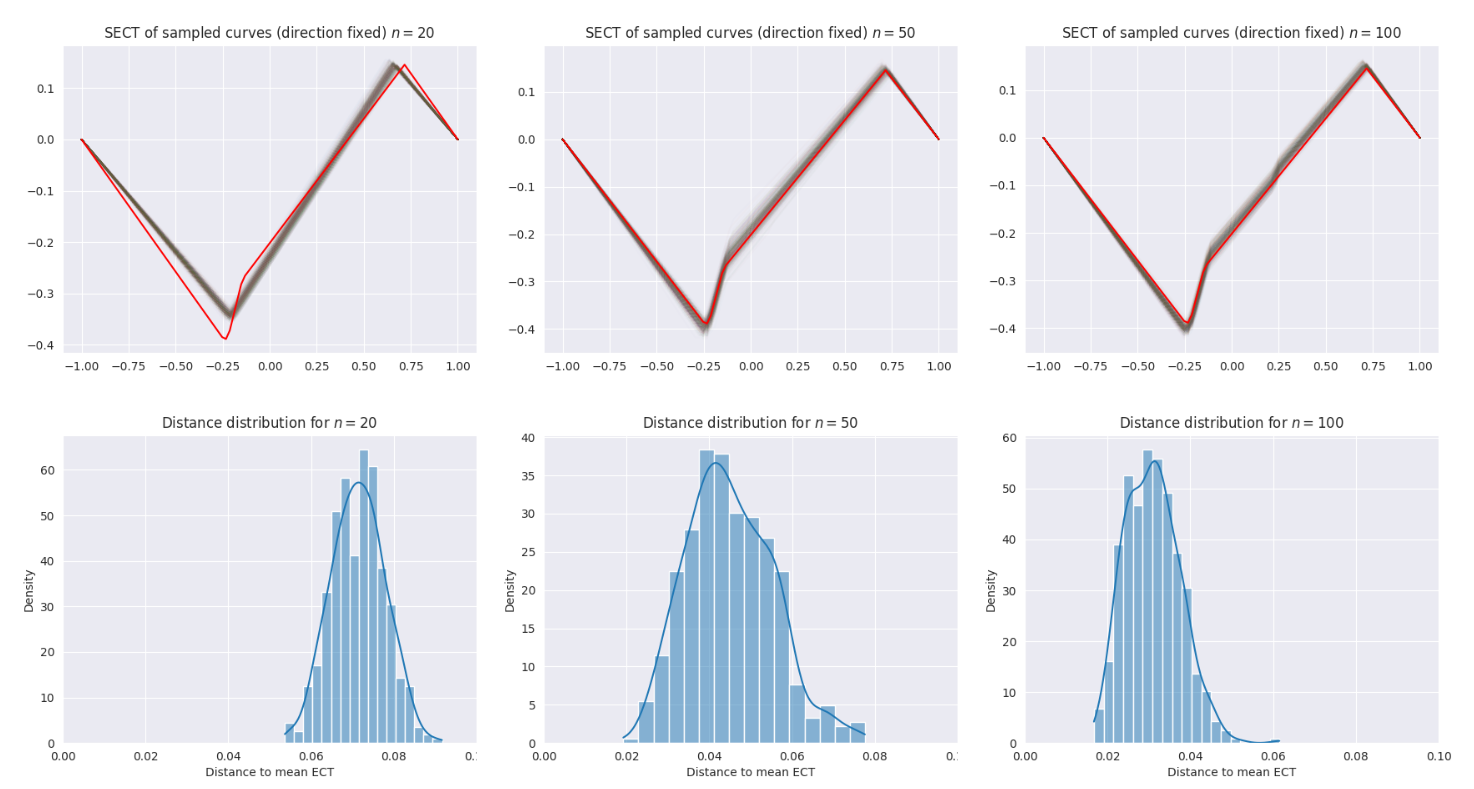}
    \caption{Top: The SECT in a fixed direction of the true shape (in red) compared to the SECTs of GPR posterior samples (in opaque blue) based on 20 (left), 50 (middle) and 100 (right) samples. The fixed direction corresponds to left-to-right in Figure \ref{fig:example_curves}. The SECTs are based on interpolations on the samples. Bottom: The distribution of the distance between the SECT of the true curve to SECTs of GPR posterior curves based on 20 (left), 50 (middle) and 100 (right) noisy samples from the underlying curve.}
    \label{fig:ect_dist}
\end{figure}

The example in this section illustrates how our estimator naturally gives rise to a posterior distribution over the space of SECT curves. We believe that there is potential to use this posterior distribution in a statistical inference or classification pipeline. Proving convergence rates for estimators like ours would help with quantifying the confidence of statistical ECT analyses.

\subsection{Characterisation of the sine-squared exponential kernel}

\begin{lemma}\label{thm:sse_bound}
For the sine-squared kernel, we have $J(S^1,d_k)<\infty$.
\end{lemma}

\begin{proof}
For the sine squared kernel $k$, the metric $d_k$ is given by 
$$d_k(s,t)=\sqrt{2-2\exp(-2\sin^2((s-t)/2)}.$$

It can be shown that $d_k$ is strongly equivalent to the angular metric $d$: let $$f(x)=\sqrt{2-2\exp(-2\sin^2(x))}.$$ Then 
$$f'(x)=\dfrac{4{e}^{-2\sin^2\left(x\right)}\cos\left(x\right)\sin\left(x\right)}{\sqrt{2-2\exp\left({-2\sin^2\left(x\right)}\right)}}.$$
In particular, $f'(x)\geq 0$ on $0\leq x\leq\pi/2$ (we can show that $\lim_{x\to0^+}f'(x)=2$ by L'Hopital's rule).
Further, $f$ is concave as it is the composition of non-decreasing concave functions. Thus, $d_k(s,t)=f(d(s,t)/2)$, we get $d\geq d_k\geq (2\sqrt{2-2e^{-2}}/\pi) d$ on $S^1$, where the first factor is $f'(0)/2$ and the second factor is the difference quotient of $f$ between 0 and $\pi/2$.

As $S^1$ is bounded and $d$ and $d_k$ are strongly equivalent, it is thus sufficient to show that $J(S^1,d)<\infty$. For $d$ and $\varepsilon>0$, we get $$N(S^1, d,\varepsilon)=\left\lceil\frac{\pi}{\varepsilon}\right\rceil\leq\frac{\pi}{\varepsilon}+1.$$
Hence,
\begin{align*}
    J(S^1, d) &= \int_0^\infty \sqrt{\log N(S^1, d,\varepsilon)}\,\mathrm{d}\varepsilon\\
    &\leq \int_0^\pi \sqrt{\log\left(\frac{\pi}{\varepsilon}+1\right)}\,\mathrm{d}\varepsilon\\
    &= \pi\int_1^\infty \frac{\sqrt{\log\left(x+1\right)}}{x^2}\,\mathrm{d}x\\
    &\leq \pi\int_1^\infty \frac{1}{x^\frac{3}{2}}\,\mathrm{d}x = \pi\left[-2x^{-\frac{1}{2}}\right]_1^\infty=2\pi<\infty.
\end{align*}
\end{proof}

\begin{lemma}\label{thm:sse_rkhs}
Define the Hilbert space $\mathcal{H}'$ of 
sequences $w_{ab}\in\RR$, $a,b\in\NN_0$, satisfying
$$\sum_{n=0}^\infty n!\sum_{\substack{a\geq0,\,b\geq0\\a+b=n}}\frac{w_{ab}^2}{C^n_a}<\infty,$$
where $C_a^n$ denotes $n$ choose $a$.
For $\{w_{ab}\}, \{v_{ab}\} \in\mathcal{H}'$, the inner-product of $\mathcal{H}'$ is given by
$$\langle \{w_{ab}\}, \{v_{ab}\} \rangle_\mathcal{H'}:=\gamma\sum_{n=0}^\infty n!\sum_{\substack{a\geq0,\,b\geq0\\a+b=n}}\frac{w_{ab}v_{ab}}{C^n_a}$$
where $\gamma>0$ is a constant. Define $V$ to be the closed subspace of sequences $\{w_{ab}\}\in\mathcal{H}'$ such that
\begin{equation}
    \sum_{(a,b)\in\NN^2}w_{ab}\cos^a(t)\sin^b(t)=0 \label{eq:subspacev}
\end{equation}
for all $t\in[0,2\pi)$. Then the Hilbert space $\mathcal{H}$ given by the functions
\begin{equation}
f(t)=\sum_{(a,b)\in\NN^2}w_{ab}\cos^a(t)\sin^b(t)\label{eq:hchar}
\end{equation}
with $\{w_{ab}\}\in V^\perp$ and inner-product induced from $\mathcal{H}'$ is isomorphic to the RKHS of the sine-squared-exponential kernel, denoted by $\mathcal{H}_k$.
\end{lemma}

\begin{proof}
By using standard trigonometric identities, we see that the sine-squared kernel is proportional (by a positive constant) to the kernel
$$k(s, t)=\exp\left(\cos(s)\cos(t)+\sin(s)\sin(t)\right).$$
Thus, by using the Taylor expansion of $\exp$, $k(\,\cdot\,,t)\in\mathcal{H}$ for all $t$ with coefficients
$$w_{ab}=\frac{C^{a+b}_a}{(a+b)!}\cos^a(t)\sin^b(t).$$
Moreover, for $f\in\mathcal{H}$ with coefficients $v_{ab}$ and fixed $t$, we get
\begin{align}
    \langle k(\,\cdot\,,t), f\rangle_\mathcal{H} &= \sum_{n=0}^\infty n!\sum_{\substack{a\geq0,\,b\geq0\\a+b=n}}\frac{C^n_a\cos^a(t)\sin^b(t)v_{ab}}{n!C^n_a}\notag \\
    &= \sum_{(a,b)\in\mathbb{N}^2}\cos^a(t)\sin^b(t)v_{ab} = f(t).\label{eq:vperp}
\end{align}
Thus, the inner-product $\langle\,\cdot\,,\,\cdot\,\rangle_\mathcal{H}$ has the reproducing property and coincides with the inner-product induced by the kernel $k$ (i.e. the inner-product of $\mathcal{H}_k$) given in Equation (\ref{eq:rep_prop}). Further, the coefficients of $k(\,\cdot\,,t)$ lie in $V^\perp$: let $\{v_{ab}\}\in V$ and let $\{w_{ab}\}$ be the coefficients of $k(\,\cdot\,,\,t)$. Then by Equation (\ref{eq:vperp}), $\langle\{v_{ab}\},\{w_{ab}\}\rangle_\mathcal{H'}=0$. As $V^\perp$ is closed as it is perpendicular to $V$, so is $\mathcal{H}$, implying that $\mathcal{H}$ is a Hilbert space. We have that $\mathcal{H}_k\subseteq\mathcal{H}$. As $\mathcal{H}_k$ is complete by definition, we get $\mathcal{H}=\mathcal{H}_k\oplus W$ for some closed subspace $W$. Let $f\in W$. Then $\langle g, f\rangle_\mathcal{H}=0$ for all $g\in\mathcal{H}_k$ and in particular $f(t)=\langle k(\,\cdot\,, t)\rangle_\mathcal{H}=0$ for all $t\in[0,2\pi)$.
Thus, $W=0$ and  $\mathcal{H}\cong\mathcal{H}_k$.
\end{proof}

\begin{lemma}\label{thm:fourier}
    Every $f\in\mathcal{H}$ is continuous and the inclusion $\mathcal{H}\hookrightarrow C(S^1,d_\infty)$ is continuous, where $C(S^1,d_\infty)$ is the space of continuous real-valued functions on $S^1$ endowed with the $\infty$-norm. Further, $\cos(nt)$ and $\sin(nt)$ are elements of $\mathcal{H}$ for all $n\in\NN$.
\end{lemma}

\begin{proof}
    Note that $\|k(\,\cdot\,,t)\|_\mathcal{H}=1$ for all $t$. Thus, by the reproducing property of $k$ and the Cauchy-Schwarz inequality, for all $f\in\mathcal{H}$ and any $t\in S^1$ we get
    $$|f(t)|=|\langle k(\,\cdot\,,t),f\rangle_\mathcal{H}|\leq\|f\|_\mathcal{H}.$$
    Hence, convergence in the $\mathcal{H}$-norm implies convergence in the $\infty$-norm. As $f$ can be written as a series of continuous functions converging in the $\mathcal{H}$-norm (c.f. Equation (\ref{eq:hchar})), it follows that $f$ is continuous. As for any $f\in\mathcal{H}$ with $\|f\|_\mathcal{H}\leq \epsilon$ and any $t\in S^1$ we have $|f(t)|\leq \epsilon$, we get that the inclusion $\mathcal{H}\hookrightarrow C(S^1,d_\infty)$ is continuous.

    Further, we can expand
    $$
    \cos(nt) = \sum_{k\text{ even}}^n (-1)^\frac{k}{2} \binom{n}{k}\cos^{n-k} (t) \sin^k (t),\qquad
    \sin(nt) = \sum_{k\text{ odd}}^n (-1)^\frac{k-1}{2} \binom{n}{k}\cos^{n-k} (t) \sin^k (t).$$
    Thus, $\cos(nt)$ and $\sin(nt)$ can be expanded as powers of $\cos$ and $\sin$ with coefficients in $\{w_{ab}\}\in\mathcal{H}'$ (as in Lemma \ref{thm:sse_rkhs}). We can project these coefficients into $V^\perp$ without changing the value of our series at any $t\in S^1$: the difference in the series we observe by subtracting from elements of $V$ from $\{w_{ab}\}$ is 0 for all $t$ (c.f. Equation (\ref{eq:subspacev})). Thus, $\cos(nt), \sin(nt) \in\mathcal{H}$ for all $n\in\NN$.
\end{proof}

\section*{Acknowledgements}

The authors thank Heather A Harrington and Vidit Nanda for fruitful discussions and helpful comments on this manuscript. Both authors are members of the Centre for Topological Data Analysis, which is funded by the EPSRC grant `New Approaches to Data Science: Application Driven Topological Data Analysis' \href{https://gow.epsrc.ukri.org/NGBOViewGrant.aspx?GrantRef=EP/R018472/1}{\texttt{EP/R018472/1}}. LM gratefully acknowledges support from the EPSRC Grant EP/R513295/1 and support from the Ludwig Institute for Cancer Research. For the purpose of Open Access, the authors have applied a CC BY public copyright licence to any Author Accepted Manuscript (AAM) version arising from this submission.

\section*{Data and Code Availability}

Data and code will be shared upon reasonable request.

\printbibliography

\end{document}